\documentclass[11pt]{amsart}

\usepackage[cp1250]{inputenc}
\usepackage[T1]{fontenc}
\usepackage[T1]{fontenc}
\usepackage{t1enc}
\usepackage{amsfonts,amsmath,amssymb,amsthm,color,amsbsy}
\usepackage{color}
\usepackage{url}
\usepackage{eucal}
\usepackage{enumerate}
\usepackage{graphicx}
\usepackage{xfrac}

\usepackage{url}
\usepackage{mathrsfs}

\usepackage{bm}%pogrubiona \nabla

\newtheorem{theorem}{Theorem}[section]
\newtheorem{lemma}[theorem]{Lemma}
\newtheorem{proposition}[theorem]{Proposition}
\newtheorem{corollary}[theorem]{Corollary}
\newtheorem{remark}[theorem]{Remark}

\newtheorem{definition}[theorem]{Definition}

\renewcommand{\geq}{\geqslant}
\renewcommand{\leq}{\leqslant}
\renewcommand{\tilde}{\widetilde}

\def\V{\mathbf{V}}
\def\ind{\mathrm{ind}_{LS}}
\def\A{\mathbf{A}}
\def\I{\mathbf{I}}
\def\G{\mathbf{G}}
\def\P{\mathbf{P}}
\def\Q{\mathbf{Q}}
\def\R{\mathbb{R}}
\def\C{\mathbb{C}}
\def\N{\mathbb{N}}

\def\lma{\lambda}
\def\part{\partial}

\def\Ker{\mathrm{Ker}\,}
\def\essinf{\mathrm{essinf}\,}
\def\esssup{\mathrm{esssup}\,}

\def\dist{\mathrm{dist}}
\def\la{\langle}
\def\ra{\rangle}

\def\vp{\varphi}

\def\F{\mathbf{F}}
\def\B{\mathbf{B}}

\def\c{\mathbb{C}}

\def\eps{\varepsilon}
\def\c{\mathrm{const.}}

\def\pr{\right )}
\def\le{\left (}
\def\part{\partial}
\def\K{\mathcal{K}}

\def\R{\mathbb{R}}
\def\N{\mathbb{N}}
\def\K{\mathbf{K}}

\def\Re{\mathrm{Re}\,}
\def\mm{\kern +2pt \raisebox{+0.5 pt}{$\shortmid$}\kern -2pt\hbox{$\multimap$}\kern +2pt}
\def\H1{H^1\le\Omega,\mathbb{R}^n\pr}
\def\L2{L^2\le\Omega,\mathbb{R}^n\pr}
\def\h10{H^1_0\le\Omega,\mathbb{R}^n\pr}

   \newcommand{\be}{\begin{equation}}
  \newcommand{\ee}{\end{equation}}

\numberwithin{equation}{section}

\author{Aleksander \'Cwiszewski, Wojciech Kryszewski}
\address{Faculty of Mathematics and Computer Sciences, Nicolaus Copernicus University,  Toru\'n, Poland}
\email{aleks@mat.umk.pl, wkrysz@mat.umk.pl}

\title[Bifurcation form infinity]{Bifurcation from infinity for elliptic problems on $\R^N$}

\date{\today}
\begin{document}
\pagestyle{myheadings}
\baselineskip15pt
\begin{abstract}
In the paper the asymptotic bifurcation of solutions to a parameterized stationary semilinear Schr\"odinger equation involving a potential of the Kato-Rellich type  is studied. It is shown that the bifurcation from infinity occurs if the parameter is an eigenvalue of the hamiltonian lying below  the asymptotic bottom of the bounded part of the potential. Thus the bifurcating solution are related to bound states of the corresponding Schr\"odinger equation. The argument relies on the use of the (generalized) Conley index due to Rybakowski and resonance assumptions of the Landesman-Lazer or sign-condition type.
\end{abstract}
\maketitle
\vspace{-7mm}

\section{Introduction}

We study a parameterized elliptic problem  \begin{equation}\label{24062015-1032}
\left\{ \begin{array}{l}
-\Delta u(x) + V(x)u(x)=\lma u(x)+f(x,u(x)),\, x\in \R^N, \lma\in\R,\\
u\in H^1(\R^N),
\end{array}\right.
\end{equation}
related to a nonlinear Schr\"odinger equation \eqref{nse} and its bound states of the form \eqref{standing}. Solutions to \eqref{24062015-1032} may also be interpreted as stationary states of the corresponding reaction-diffusion equation \eqref{15052017-2001}.\\
\indent We are interested in a characterization of asymptotic bifurcation for \eqref{24062015-1032}.
\begin{definition}
{\em A parameter $\lma_0\in \R$ is a point of {\em bifurcation from infinity} or {\em asymptotic bifurcation} of solutions to \eqref{24062015-1032} if there exists a sequence $(\lma_n, u_n)_{n=1}^\infty$  such that $\lma_n\to \lma_0$, $u_n\in H^1(\R^N)$ is a weak solution of (\ref{24062015-1032}) with $\lma=\lma_n$ for each $n\geq 1$, and $\|u_n\|_{H^1} \to +\infty$.}
\end{definition}
\indent The study of asymptotic bifurcation, apparently started by M. Krasnoselskii \cite{Kras},  who introduced the notion of an asymptotically linear operator, and P. Rabinowitz \cite{Rabinowitz}, as well as the study of bifurcation from zero (i.e. from the zero solution), have been conducted by numerous authors from both the abstract and application viewpoints (e.g. by Toland, Dancer, Mawhin,  Schmitt, Ward and many others; see e.g. \cite{Toland,Dancer,Ward1,Mawhin,Schmitt_Wang}). These problems are related since it is often possible to adapt ideas and techniques coming from the study of bifurcation from zero to asymptotic bifurcation; this was effectively employed by Toland in \cite{Toland} and in \cite{Rabinowitz,Stuart1} via the so-called {\em Toland inversion}.  Most of applications to PDEs were concerned with bifurcation and multiplicity of solutions to elliptic problems of the form $-\Delta u=\lambda u+f(x,u)$ on a bounded domain $\Omega\subset \R^N$ together with various boundary conditions (see e.g. \cite{Arrieta,Gamez,DLi}). A careful analysis of interactions (i.e. crossing) of $\lambda$ with the (purely discrete) spectrum of $-\Delta$ subject to the boundary condition along with appropriate behavior of $f$ such as, for instance, the so-called `sign condition', leads to the existence and multiplicity of solution. In \cite{Mawhin} (see also \cite{Chiap,Mawhin_Chiap,Schmitt_Wang}) it was pointed out that a {\em condition of the Landesman-Lazer type} could substitute the sign condition. The topological tools used depend on the parity of the crossed eigenvalue of $-\Delta$: roughly speaking topological degree techniques are exploited if $\lambda $ crosses an eigenvalue of odd multiplicity while variational methods are used in the case of even multiplicity.\\
\indent The problem of bifurcation of solutions to elliptic problems on $\R^N$ is not that well-recognized. A detailed study of bifurcation from zero is given e.g. in \cite{Evequoz,Stuart,Rabier}, while questions of asymptotic bifurcation were dealt with in \cite{Genoud}, \cite{Stuart1} (see also the references therein) and \cite{Kr-Sz}. An important issue of the spectral theory of elliptic equations on $\R^N$, as opposed to its counterpart on bounded domains, is that the spectrum of $-\Delta+V(x)$ is not discrete in general and, depending on the potential, may be quite complicated. Results from \cite{Genoud,Stuart1,Kr-Sz} show that the existence of asymptotic bifurcation at an eigenvalue $\lambda_0$ relies on the appropriate relationship between $\lambda_0$, $f$ and the essential spectrum of $-\Delta+V(x)$ inasmuch as bound states of the Schr\"odinger equation correspond to energies below the bottom of the essential spectrum.\\
\indent Let us now present the standing assumptions. As concerns the potential generating the hamiltonian $${\A}:=-\Delta+V(x)$$ we assume that
\begin{align}\label{postac poten}  &V\in L^\infty(\R^N)+L^p(\R^N),\; \text{i.e.}\;\; V=V_\infty+V_0,\; \text{where}\\
\label{24062015-1047} &V_\infty \in L^\infty(\R^N)\;\; \text{and}\;\; V_0\in L^p(\R^N),\; p\geq 2\;\; \text{if}\;\; N=1,\; p>2\;\; \text{if}\;\; N=2\;\; \text{and}\;\; p\geq N\;\; \text{for}\;\; N\geq 3,\end{align}
and, as concerns the nonlinear interaction term, we assume that $f:\R^N\times \R\to\R$ is a Carath\'eodory  function such that
\begin{align}\label{17022016-1449}
&|f(x ,u)|\leq m(x) \mbox{ for all } u\in\R \mbox{ and a.e. } x\in\R^N,\\
\label{17022016-1450}
&|f(x,u)-f(x,v)|\leq l(x)|u-v| \mbox{ for all }  u,v\in\R \mbox{ and for a.e. } x\in\R^N,
\end{align}
where  $m\in L^2(\R^N)$, $l=l_0+l_\infty$ with $l_0$ satisfying  (\ref{24062015-1047}) (with $l_0$ instead of $V_0$) and $l_\infty\in L^\infty(\R^N)$.
\begin{remark}\label{pocz} {\em Observe that $V$ belongs the the so-called {\em Kato class} of potentials $K_N$ considered by Aizenman and Simon (see \cite[A.2]{Simon}) since, $L^r(\R^N)\subset K_N$ whenever $r\geq 2$ with $r>N/2$,  $N\geq 2$, or a slightly more general class considered in \cite{Hempel}. If, for instance, $V$ is the {\em Coulomb type} potential, i.e. $V(x):=c/|x-x_0|^{\alpha}$ for $x\neq x_0$, where $x_0\in\R^N$, $c\in\R$ and $\alpha\in [0,\sfrac{1}{2})$ if $N=1$, $\alpha\in [0,1)$ for $N=2$ and $\alpha\in [0,1)$ for $N\geq 3$, then $V$ satisfies conditions \eqref{postac poten} and \eqref{24062015-1047} since one may take $V_0=\chi V$ and $V_\infty=(1-\chi)V$, here $\chi$ is the characteristic function of the unit ball in $\R^N$ around $x_0$.
\hfill $\square$}
\end{remark}
\indent Since  $\lim_{|s|\to +\infty}f(x,s)/s=0$ for $x\in\R^N$, one expects that, as in the classical situation (see e.g. \cite{Rabinowitz}), if $\lambda$ approaches an eigenvalue of ${\A}$, then solutions to \eqref{24062015-1032} bifurcate from infinity  as the result of a  produced resonance phenomenon. Indeed: as we shall see in Theorem \ref{s16022016-1655}, the necessary condition for $\lma_0$ lying beyond the essential spectrum of the hamiltonian for inducing asymptotic bifurcation is that $\lambda_0\in\sigma_p({\A})$ the point spectrum of the hamiltonian. Conversely, if $\lambda_0$ is an isolated eigenvalue of {\em odd} multiplicity, then the asymptotic bifurcation occurs. In order to provide sufficient conditions for asymptotic  bifurcation from an isolated eigenvalue of {\em even} multiplicity, one needs to impose additional assumptions concerning the behavior  of $f$ at infinity: the so-called {\em Landesman-Lazer type} or {\em strong resonance} conditions.\\
\indent The {\em Landesmann-Lazer type conditions} state that either
$$
\left\{
\begin{array}{c}
\check{f}_+(x) \geq 0 \ \mbox{ and } \ \hat{f}_-(x)\leq 0 \ \mbox{ for a.e. } \ x\in\R^N,\\
\mbox{there is a set of positive measure on which none of}\;\; \check{f}_+ \mbox{ and } \hat{f}_- \mbox{vanishes},
\end{array}
\right.
\leqno{(LL)_+}
$$
or
$$
\left\{
\begin{array}{c}
\hat{f}_+(x)\leq 0 \ \mbox{ and } \ \check{f}_- (x)\geq 0 \ \mbox{ for a.e. } \ x\in\R^N,\\
\mbox{there is a set of positive measure on which none of }\; \hat{f}_+ \mbox{ and }  \check{f}_- \mbox{ vanishes,}
\end{array} \right.
 \leqno{(LL)_-}
$$
where $\hat{f}_\pm(x):= \limsup_{s \to \pm\infty} f(x,s)$ and  $\check{f}_\pm (x):= \liminf_{s \to \pm\infty} f(x,s)$ for $x\in\R^N$.
\begin{remark}\label{LL} {\em Conditions of this type has been considered by many authors; see e.g.
\cite{Fonda} for a relatively up-to-date survey. Observe (see also the proof of Lemma \ref{25092017-1854}) that $(LL)_+$ (resp. $(LL)_-$), together with the so-called {\em unique continuation property}, imply that
\be\label{ll}\int_{\R^N}(\check{f}_+\varphi^+-\hat{f}_-\varphi^-)\,dx>0\;\;
\left(\text{resp.}\;\; \int_{\R^N}(\hat{f}_+\varphi^+-\check{f}^-\varphi^-)\,dx<0\right)\ee
for any eigenfunction $\varphi$ of the hamiltonian ${\A}$ and $\vp^\pm=\max\{0,\pm\vp\}$. Clearly \eqref{ll} is the classical Landesman-Lazer condition (see e.g. \cite[eq. $(LL)$]{Fonda}); one can easily check by proof-inspection that each of the conditions stated in \eqref{ll} is actually sufficient for our purposes.\hfill $\square$ }
\end{remark}
\indent The so-called {\em sign conditions} or {\em strong resonance conditions}
 are fulfilled if
 $k_\pm(x):=\lim_{s\to \pm\infty} sf(x,s)$ exists for a.a. $x\in\R^N$, $k_\pm\in L^\infty(\R^N)$ and either
$$\left\{
\begin{array}{c} sf(x,s)\geq 0\;\;\text{for a.a.}\; x\in\R^N\; \text{and all}\; s\in\R,\\
\mbox{and there is a set of positive measure on which } k_\pm\; \text{is positive,}\end{array}
\right.
\leqno{(SR)_+}
$$
or
$$\left\{
\begin{array}{c}
sf(x,s)\leq 0\;\;\text{for a.a.}\; x\in\R^N\; \text{and all}\; s\in\R,\\
\mbox{and there is a set of positive measure on which } k_\pm\;  \text{is negative}.\end{array}
\right.
\leqno{(SR)_-}
$$

As we shall see (comp. Lemma \ref{25092017-1854}) both assumption $(LL)_\pm$ and $(SR)_\pm$ lead to the geometric condition \eqref{25092017-2303} concerning inward (or outward) behavior of the nonlinearity with respect to eigenspaces of ${\A}$. Such conditions were already studied on an abstract level in \cite[Eq. (2.3) or (2.4)]{Mawhin}, \cite{Cesari} and \cite{Kokocki}. A discussion of some other resonance conditions and their role is provided in \cite{Bartolo}.

Our main result is as follows. Let
\be\label{def alfa}\alpha_\infty:=\lim_{R\to\infty}\essinf_{|x|\geq R}V_\infty(x),\ee
be the {\em asymptotic bottom of the potential $V_\infty$}.
\begin{theorem}\label{24062015-1056} Suppose that $\lambda_0\in\sigma(\A)$. If either\\
\indent {\em (i)} $\lambda_0$ is an isolated eigenvalue of odd multiplicity; or \\
\indent {\em (ii)} $\lma_0<\alpha_\infty$ {\em (\footnote{We shall see that this implies that $\lambda_0$ is an isolated eigenvalue of finite multiplicity.})} and  one of conditions $(LL)_\pm$  or $(SR)_\pm$ holds,\\
then $\lma_0$ is a point of bifurcation from infinity for {\em (\ref{24062015-1032})}.
\end{theorem}

\begin{remark} {\em (1) It is clear that if $(\lambda_n,u_n)$ is a sequence bifurcating form infinity at $\lambda_0$, then $u_n\in H^2(\R^N)$ and  $\|u_n\|_{H^2}\to +\infty$. In Theorem \ref{s16022016-1655} we show that under the assumptions of the above theorems also both sequences $(\|u_n\|_{L^2})$ and $(\|\nabla u_n\|_{L^2})$ tend to infinity; moreover these sequences have the same growth rate.\\
\indent (2) Theorem \ref{24062015-1056} complements and  generalizes results concerning the asymptotic bifurcation for equations of the form \eqref{24062015-1032}  from \cite{Stuart1} and \cite{Kr-Sz}. In \cite{Kr-Sz} problem \eqref{24062015-1032} was studied when $V\in L^\infty(\R^N)$ (i.e., $V_0\equiv 0$) and under hypotheses which, together with the ansatz $(f_4)$ (see \cite[p. 415]{Kr-Sz}), imply  our standing assumptions with one important difference in comparison to \eqref{17022016-1449}: in the setting of \cite{Kr-Sz}, the bounding function $m\in L^\infty(\R^N)$. In \cite{Stuart1} a similar problem is very thoroughly investigated with $f(x,u)=h(x)+\tilde f(u)$, where $h\in L^2(\R^N))$ and $\tilde f(u)/u\to 0$ as $|u|\to +\infty$ (see the assumption $(G)$ in \cite{Stuart1}). In both papers the asymptotic bifurcation occurs at an eigenvalue $\lambda_0$ of ${\A}$ provided the distance $\dist(\lambda_0,\sigma_e({\A}))$ of $\lambda_0$ to $\sigma_e({\A})$, the essential spectrum of the hamiltonian, is larger than the Lipschitz constant of the nonlinearity $g$ (in \cite{Stuart1} a bit more restrictive bound is necessary). Such a condition was also implicitly contained in \cite[Assumption D]{Dancer}. If the multiplicity of $\lambda_0$ is odd, then the proofs from \cite{Stuart1,Kr-Sz} use the degree theory (via the Toland inversion in \cite{Stuart1}), while for an eigenvalue of even multiplicity the existence of asymptotic bifurcation in \cite{Kr-Sz} relies on a variational approach based on the Morse theory. In \cite{Genoud} the principal eigenvalue (being simple) of the linearization at infinity is shown to be a point of asymptotic bifurcation and the result is obtained by the Toland inversion. \\
\indent In our approach the physically relevant unbounded part $V_0$ of the potential is not trivial, but, at least in case the multiplicity of $\lambda_0$ is even, we need that $\lambda_0<\alpha_\infty$ which, as we shall see, implies that $\lambda_0$ lies below the bottom of $\sigma_e({\A})$;  observe that the spectrum $\sigma_e(-\Delta+V_\infty)\subset [\alpha_\infty,\infty)$. We do not require any relations of the distance $\dist(\lambda_0,\sigma_e({\A}))$ with the Lipschitz constant, but instead we make use of the estimate \eqref{17022016-1449}. If $V_0\neq 0$ (making $V$ look like a potential well) is sufficiently deep and steep, then $\sigma({\A})\cap (-\infty,\alpha_\infty)\neq\emptyset$ (this holds for instance if $V$ is the Coulomb type potential from Remark \ref{pocz}; see also eg. \cite[Theorem XIII.6]{ReedSimon} and \cite{Schmud}).\\
\indent (3) Our attitude to the first part of Theorem \ref{24062015-1056} is based on the Leray-Schauder degree theory; in this context condition \eqref{17022016-1450} is not necessary since the continuity of the Nemytskii operator generated by $f$ is sufficient. In the second part we shall rely on the Conley index theory applied to the semiflow generated by the parabolic equation
\begin{equation}\label{15052017-2001}
u_t=\Delta u-V(x)u+\lambda u+f(x,u),\;\; x\in\R^N,\; u\in\R,\; t>0,
\end{equation}
related to \eqref{24062015-1032}.  We shall show that  assumptions  imply that this semiflow is well-defined and its  Conley indices `at infinity' change when the parameter $\lambda$ crosses $\lambda_0$. To meet the quite demanding requirements concerning compactness issues (i.e. the so-called {\em admissibility} of the semiflow with respect to bounded sets) we  adopt some ideas of Prizzi \cite{Prizzi-FM,Prizzi}. The use of  the (generalized) Conley type index of Rybakowski \cite{rybakowski-TAMS} in the context of bifurcation has been started by Ward \cite{Ward1,Ward2} and applied for elliptic problems on bounded domains. Quite recently this approach has been thoroughly complemented and expanded in \cite{DLi} (see also the rich bibliography therein) and applied to bifurcation problems on bounded domains. To the best of our knowledge the present paper is the first one to employ Conley index to the asymptotic bifurcation for elliptic problems in $\R^N$.\hfill $\square$}\end{remark}
\indent Let us now discuss the physical context of the studied problem. We consider the externally driven {\em nonlinear Schr\"odinger equation} of the form
\be\label{nse}i\psi_t=-\Delta\psi+V(x)\psi-W'(x,\psi),\ee
and its bound states, i.e. {\em wave-functions} $\psi:[0,+\infty)\times\R^N\to\C$ that vanish at infinity; here $V$ satisfies assumptions \eqref{postac poten} and \eqref{24062015-1047},  $W:\R^N\times\C\to\R$ and $W'(x,z):=\frac{\part}{\part z_1}W(x,z)+i\frac{\part}{\part z_2}W(x,z)$, $x\in\R^N$, $z=z_1+iz_2$. One usually assumes that $W$ depends on $x\in\R^N$ and $|z|$ only, i.e. $W(x,z)=H(x,|z|)$ where $H:\R^N\times [0,+\infty)\to\R$ has the form
$$H(x,s)=\int_0^sh(x,\xi)\,d\xi,\;\;x\in\R^N,\;s\geq 0,$$
and $h:\R^N\times [0,+\infty)\to\R$ is a Carath\'eodory function satisfying conditions analogous to \eqref{17022016-1449} and \eqref{17022016-1450}. Therefore for all $x\in\R^N$
$$W'(x,z)=h(x,|z|)\frac{z}{|z|}\;\;\text{for}\;\;z\in\C\setminus\{0\},\;\; W'(x,0)=0.$$
Problems concerning \eqref{nse} play an important role in different physical contexts, especially in the description of  macroscopic quantum systems like, for instance, plasma physics, nonlinear optics and others -- see e.g. \cite{Feng}, \cite{Sulem}. For appropriate choice of $h$ the equation \eqref{nse} has {\em standing wave} solutions, i.e. satisfying the ansatz
\be\label{standing}\psi(t,x)=e^{-i\lambda t}u(x),\;\;t\geq 0,\;x\in\R^N,\ee
with the time-independent profile $u\in H^1$ and $\lambda\in\R$. Substituting \eqref{standing} into \eqref{nse} and putting for $x\in\R^N$ and $u\in\R$
\be\label{wyr f}f(x,u):= h(x,|u|)\frac{u}{|u|}\;\;\text{if}\; u\neq 0,\;\;f(x,0)=0,\ee we get \eqref{24062015-1032} along with our standing assumptions; clearly any solution $(\lambda,u)\in\R\times H^1$ gives via \eqref{standing} a bound state $\psi$ for \eqref{nse}.\\
\indent  The {\em energy} (see \cite{Benci}) of a wave-function $\psi$ satisfying \eqref{nse}, given by $$E(\psi):=\frac{1}{2}\int_{\R^N}(|\nabla\psi|^2+V(x)|\psi|^2)\,dx-\int_{\R^N} W(x,\psi)\,dx$$
is time invariant and, in case \eqref{standing},
$$E(\psi)=\frac{1}{2}\int_{\R^N}(|\nabla u|^2+V(x)u^2)\,dx-\int_{\R^N}H(x,|u|)\,dx.$$
\begin{theorem}\label{physics} Suppose that $\lma_0<\alpha_\infty$, where $\alpha_\infty$ is given by \eqref{def alfa}, $\lambda_0\in\sigma(-\Delta+V)$ and  one of the following conditions is satisfied:\\
\indent $(i)_+$ for a.a. $x\in\R^N$, $\check{h}(x):=\liminf_{\xi\to +\infty}h(x,\xi)\geq 0$ and $\check{h}$ is positive on a set of positive measure;\\
\indent $(i)_-$ for a.a. $x\in\R^N$, $\hat{h}(x):=\limsup_{\xi\to +\infty}h(x,\xi)\leq 0$ and $\hat{h}$ is negative on a set of positive measure;\\
\indent $(ii)_+$ for a.a. $x\in\R^N$ and all $\xi\geq 0$, $h(x,\xi)\geq 0$ and $\lim_{\xi\to +\infty}\xi h(x,\xi)$ is positive on a set of positive measure;\\
\indent $(ii)_-$ for a.a. $x\in\R^N$ and all $\xi\geq 0$, $h(x,\xi)\leq 0$ and $\lim_{\xi\to +\infty}\xi h(x,\xi)$ is negative on a set of positive measure.\\
\indent Then there is a sequence $(\psi_n)$ of bound states of \eqref{nse} of the form
$\psi_n(t,x)=e^{-i\lambda_n t}u_n(x)$ for $t\geq 0$, $x\in\R^N$, where $\lambda_n\in\R$, $u_n\in H^1$ for all $n\geq 1$, $\lambda_n\to\lambda_0$ and $\|u_n\|_{H^1}\to +\infty$. If $\lambda_0\neq 0$, then $|E(\psi_n)|\to +\infty$.
\end{theorem}
\noindent {\em Proof:} It is easy to see that if $f$ is given by \eqref{wyr f}, then condition
$(i)_\pm$ (resp. $(ii)_\pm$) implies $(LL)_\pm$ (resp. $(SR)_\pm$); hence, in view of Theorem \ref{24062015-1056}, there is a sequence $(\lambda_n,u_n)$ of solutions to  \eqref{24062015-1032}, yielding the existence of the required sequence of bound states. Observe that
$$E(\psi_n)=\frac{1}{2}\left(\lambda_n\|u_n\|^2_{L^2}+\int_{R^N}
(h(x,|u_n|)|u_n|-2H(x,|u_n|))\,dx\right)\geq
\frac{1}{2}\lambda_n\|u_n\|^2_{L^2}-2\|m\|_{L^2}\|u_n\|_{L^2}\to +\infty$$
when $\lambda_0>0$ and $E(\psi_n)\to-\infty$ if $\lambda_0<0$.\hfill $\square$

\indent The paper is organized as follows. Section 2 is devoted to basic notation and a brief exposition of the Conley index theory. In Section 3  we construct the semiflow related to the considered problem, study its basic properties such as continuity and admissibility; we also recall a linearizaton method to compute the Conley index of the set of bounded trajectories. Section 4 deals with necessary conditions as well as further properties of bifurcating sequences. Finally, Section 5 is devoted to the proof of the main results.

\section{Preliminaries}
 By $L^p(\Omega)$, $1\leq p\leq \infty$, and $H^{k}(\Omega)$, $k\in\N$, we denote the standard Lebesgue and Sobolev spaces on an open domain $\Omega\subset \R^N$, $N\geq 1$, with their standard norms and inner products. For brevity, in the sequel we will write $L^p$ or $H^k$ instead of $L^p(\R^N)$ and $H^k(\R^N)$.\\
 \indent If $(X,A)$ is a topological pair with a closed and nonempty $A\subset X$, then $X/A$ denotes the quotient space, obtained by collapsing the subset $A$ to a point $[A]$. Pointed spaces $(X,x_0)$ and $(Y,y_0)$ are {\em homotopy equivalent} or have {\em the same homotopy type} if there are pointed maps $f\colon (X,x_0)\to (Y,y_0)$ and $g\colon (Y,y_0)\to (X,x_0)$ such that $f\circ g$ (resp. $g\circ f$) is homotopic to the identity on $(Y,y_0)$ (resp. on $(X,x_0)$). The homotopy class represented by a space $(X,x_0)$ is denoted by $[(X,x_0)]$.

\subsection{Conley index due to Rybakowski}
We shall briefly recall a version of the Conley index due to Rybakowski (see \cite{rybakowski} or \cite{rybakowski-TAMS}). Let $\Phi\colon [0,+\infty)\times X\to X$ be a  semiflow on a complete metric space $X$. We write $\Phi_t(x):=\Phi(t,x)$ and $\Phi_{[0,t]}(x):=\{\Phi_s(x)\mid 0\leq s\leq t\}$ for $t\geq 0$, $x\in X$. A continuous $u\colon J\to X$, where $J\subset \R$ is an interval, is a {\em solution of $\Phi$} if $u(t+s) = \Phi_t(u(s))$ for all $t\geq 0$ and $s\in J$ such that $t+s\in J$. If, in addition  $0\in J$ and $u(0)=x$, then $u$ is  a {\em solution through} $x$.\\
\indent If $a\in\R$ and $u\colon [a,+\infty)\to X$ is a solution of $\Phi$, then the {\em $\omega$-limit set} of $u$ is defined by
$$\omega(u):=\{x=\lim_{n\to\infty} u(t_n)\mid t_n\geq a,\; t_n\to +\infty\};$$
if $u\colon (-\infty, a]\to X$ is a solution of $\Phi$, then the {\em $\alpha$-limit set} of $u$ is defined by
$$
\alpha (u):=\{x=\lim_{n\to\infty} u(t_n)\mid t_n\leq a,\;t_n\to -\infty\}.$$
Note that both sets $\omega(u)$ and $\alpha(u)$ are closed.\\
\indent Let $N\subset X$. We define the \emph{invariant part} $\mathrm{Inv}_\Phi (N)$ of $N$ by
$$x\in \mathrm{Inv}_\Phi (N)\,\Longleftrightarrow\,\text{there is a solution}\;u\colon \R\to N\;\text{through}\;x.$$
A set $K\subset X$ is a {\em $\Phi$-invariant} or  {\em invariant} (w.r.t. $\Phi$) if  $\mathrm{Inv}_\Phi (K)=K$. A set $K$ is an {\em isolated invariant} if there exists an {\em isolating neighborhood} of $K$, i.e. $N\subset X$ such that  $K=\mathrm{Inv}_\Phi (N) \subset \mathrm{int}\,N$.\\
\indent  A set $N\subset X$ is \emph{$\Phi$-admissible} or {\em admissible} (w.r.t. $\Phi$) if, for any sequences $(t_n)$ in $[0,+\infty)$, $(x_n)$ in $X$ such that $t_n\to +\infty$ and $\Phi_{[0,t_n]} (x_n) \subset N$, the sequence of end-points $\left( \Phi_{t_n}(x_n) \right)$ has a convergent subsequence. It is easy to see that if $N\subset X$ is $\Phi$-admissible, then the invariant part $\mathrm{Inv}_\Phi (N)$ is compact.

Suppose that $\{\Phi^\lambda\}_{\lambda\in \Lambda}$, where $\Lambda$ is a metric space, is a family of semiflows on $X$. This family is {\em continuous} if the map $[0,+\infty)\times X\times \Lambda\ni (t,x,\lambda)\mapsto\Phi^\lambda_t(x)$ is continuous. A set $N\subset X$ is
{\em admissible} w.r.t. $\{\Phi^\lambda\}$ if, for any sequences $(t_n)$ in $[0,+\infty)$, $(x_n)$ in $X$ and $(\lambda_n)$  such that $t_n\to +\infty$, $\lambda_n\to\lambda_0$ in $\Lambda$ and $\Phi^{\lambda_n}_{[0,t_n]}(x_n)\subset N$, the sequence $(\Phi^{\lambda_n}_{t_n}(x_n))$ has a convergent subsequence.

Let ${\mathcal I}(X)$ be the family of all pairs $(\Phi, K)$, where $\Phi$ is a semiflow on $X$ and a set  $K\subset X$ is  isolated invariant w.r.t. $\Phi$  having a  $\Phi$-admissible isolating neighborhood. If $(\Phi,K)\in {\mathcal I}(X)$, then the Conley {\em homotopy index} $h(\Phi,K)$ of $K$ relative to $\Phi$ is defined by
$$h(\Phi, K):=[(B/B^-, [B^-])],$$
where $B$ is an {\em isolating block} of $K$ (relative to $\Phi$; see \cite{rybakowski} for the details) with the {\em exit set} $B^- \neq \emptyset$; if $B^-=\emptyset$ we put
$h(\Phi, K):=[(B\cup \{ a\}, a)]$ where $a$ is an arbitrary point out of $B$. In particular, $h(\Phi, \emptyset)=\overline 0$ where $\overline 0:=[(\{a\},a)]$.

Let us enumerate several important properties of homotopy index:
\begin{enumerate}
\item[(H1)]
for any $(\Phi, K)\in {\mathcal I}(X)$, if $h(\Phi, K)\neq \overline{0}$, then $K\neq \emptyset$;
\item[(H2)] if $(\Phi, K_1), (\Phi, K_2)\in {\mathcal I}(X)$ and $K_1\cap K_2=\emptyset$, then $(\Phi, K_1\cup K_2)\in {\mathcal I}(X)$
and $h(\Phi, K_1\cup K_2) = h(\Phi, K_1)\vee h(\Phi, K_2)$;
\item[(H3)] for any $(\Phi_1, K_1)\in {\mathcal I}(X_1)$ and $(\Phi_2, K_2)\in {\mathcal I}(X_2)$, $(\Phi_1\times \Phi_2, K_1\times K_2)\in {\mathcal I} (X_1\times X_2)$
and $h(\Phi_1\times \Phi_2, K_1\times K_2) = h(\Phi_1, K_1) \wedge h(\Phi_2, K_2)$;
\item[(H4)] if the family of semiflows $\{\Phi^\lambda\}_{\lambda\in [0,1]}$ is continuous and there exists an admissible (with respect to this family) $N$ such that $K_\lambda = \mathrm{Inv}_{\Phi^\lambda} (N) \subset \mathrm{int}\ N$, $\lambda\in [0,1]$, then
$$
h(\Phi^0, K_0) = h(\Phi^1, K_1).
$$
\end{enumerate}
\indent In a linear case the following formula for computation of the Conley index is used.
\begin{theorem} \label{11052015-1744}{\em (See \cite[Ch. I, Th. 11.1]{rybakowski})}
Assume that a $C_0$ semigroup $\{T(t)\}_{t\geq 0}$ of bounded linear operators on a Banach space $X$ is hyperbolic {\em (}see e.g. \cite[Def. V.1.14]{Engel}{\em )}. If the dimension $\dim X_u=k$ of the unstable subspace $X_u$ {\em (\footnote{The unstable space $X_u$ is equal to
$\Ker P$, where $P$ is the spectral projection corresponding to $\{\lambda\in\sigma(T(t_0))\mid |\lambda|<1\}$ for some $t_0>0$, or the closed subspace in $X$ corresponding $\{\lambda\in\sigma(A)\mid\Re\lambda<0\}$, where $A$ is the generator of $\{T(t)\}$.})} is finite, then $\Phi\colon [0,+\infty) \times X\to X$, given by $\Phi(t,x):=T(t)x$ for $x\in X$ and $t\geq 0$, is a semiflow on $X$, $\{ 0\}$ is the maximal bounded invariant set with respect to $\Phi$, $(\Phi, \{ 0 \})\in {\mathcal I}(X)$ and $h(\Phi, \{ 0\})=\Sigma^k$ where $\Sigma^k=[(S^k, \overline s)]$  is the homotopy type of the pointed $k$-dimensional sphere.\hfill $\square$
\end{theorem}

\section{Admissibility and compactness properties of semiflow}

Let us consider problems \eqref{24062015-1032} in its  abstract form
\be\label{pr1} (\A-\lambda\I)u=\F(u),\;\;u\in H^2,\;\lambda\in\R,\ee
where $\I$ is the identity on $L^2$,
with the linear operator ${\A}:D({\A})\subset L^2 \to L^2$ given by
\begin{align}\label{operator 0}
&D(\A):=H^2(\R^N),\;\; \A:=\A_0+\V_0+\V_\infty,\;\; \text{where:}\\
\label{operator 1}&\A_0u:=-\Delta u,\;\;\text{i.e.,}\;\; \A_0u:=-\sum_{j=1}^N\frac{\part^2u}{\part x_j^2}\;\; \text{for}\;\; u\in D(\A_0)=D(\A),\\
\label{operator 2}&\V_\infty u:=V_\infty\cdot u\;\; \text{for}\;\; u\in D(\V_\infty):=L^2\;\; \text{and}\\
\label{operator 2'}&{\V}_0u=V_0\cdot u\;\; \text{for}\;\; u\in D({\V}_0):=L^q,\;\; \text{where}\;\;q\;\;\text{is given by \eqref{pocz1} below};\end{align}
and ${\F}:H^1\to L^2$ is the superposition operator generated by $f$, i.e.:
\be\label{nielin}{\F}(u):= f(\cdot,u(\cdot)),\;\; \text{for}\;\;  u\in L^2.\ee
\indent Let us discuss the above abstract setting.
\begin{remark}\label{spectra} {\em (1) By \cite[Th. 7.3.5]{Pazy}, ${\A}_0$ is  self-adjoint and sectorial. Clearly ${\V}_\infty$ is a bounded linear operator. By \cite[Proposition III.1.12]{Engel} ${\A}_0+{\V}_\infty$, defined on $D({\A}_0+{\V}_\infty)=D({\A})$, is sectorial, too. By the Kato-Rellich theorem (see \cite[Theorem 8.5]{Schmud}) it is self-adjoint. It is also clear that
$$s_\infty:=\inf\sigma({\A}_0+{\V}_\infty)=\inf_{u\in H^1,\,\|u\|_{L^2}=1}\int_{\R^N}(|\nabla u|^2+V_\infty(x)u^2)\,dx,$$
i.e. $\sigma({\A}_0+{\V}_\infty)\subset [s_\infty,+\infty)$. In view of the Persson theorem \cite[Theorem 2.1]{Persson} we have
that
$$s_\infty^*:=\inf\sigma_e({\A}_0+{\V}_\infty)=\lim_{R\to\infty}\inf\left\{\int_{\R^N}(|\nabla u|^2+V_\infty(x)u^2)\,dx\mid u\in C^\infty_0(\{|x|\geq R\}),\,\|u\|_{L^2}=1\right\}.$$
It is immediate to see that $\alpha_\infty\leq s^*_\infty$. Therefore
\be\label{def alfa 1}\sigma_e({\A}_0+{\V}_\infty)\subset [\alpha_\infty,+\infty).\ee
At most instances
$\alpha_\infty<s^*_\infty$ (see \cite{Persson}); if, however, $\lim_{R\to\infty}\esssup_{|x|\geq R}|V_\infty(x)-\alpha_\infty|=0$, then $\sigma_e({\A}_0+{\V}_\infty)=[\alpha_\infty,+\infty)$. \\
\indent (2) Let $p$ be as in \eqref{24062015-1047} and let
\be\label{pocz1}q:=\frac{2p}{p-2}\;\; \text{if}\;\; p>2,\;\; q:=\infty\;\; \text{for}\;\; p=2.\ee
Observe that, in view of the Sobolev embeddings (see \cite[Theorem 4.12]{Adams}), our assumptions imply that for any $N\geq 1$, $H^1\hookrightarrow L^q$ (continuous embeddings) and, in view of the Rellich-Kondrachov theorem (see \cite[Theorem 6.3]{Adams}), $H^2(\Omega)$ is compactly embedded in $L^q(\Omega)$ provided $\Omega\subset\R^N$ is a smooth bounded domain.\\
\indent (3)  By the above, $H^1\hookrightarrow D({\V}_0)=L^q$. In view of the H\"older inequality ${\V}_0$ is well-defined and, as the operator $L^q\to L^2$, continuous. It is symmetric, hence, closable. In view of Lemma \ref{zw} below, ${\V}_0$ is relatively $({\A}_0+{\V}_\infty)$-compact.  Therefore, by \cite[Corollary III.2.17 (ii)]{Engel}, ${\A}$ is sectorial and, in view of \cite[Proposition 8.14 (ii), Theorem 8.5]{Schmud}, ${\A}$ is self-adjoint; see also \cite[Corollary XIII.4.2]{ReedSimon}. Hence $\sigma({\A})\subset \R$.\\
\indent (4) The relative compactness of ${\V}_0$  w.r.t. ${\A}_0+{\V}_\infty$ implies, in view of the Weyl theorem (see e.g. \cite [Theorem 1.4.6]{Schechter} or \cite[Theorem 8.15]{Schmud}) and \eqref{def alfa 1}, that
\be\label{spectra 1}\sigma_e({\A})=\sigma_e({\A}_0+{\V}_\infty)\subset [\alpha_\infty,+\infty).\ee
Therefore
$\sigma({\A})\cap (-\infty,\alpha_\infty)$ is contained in the discrete part of the spectrum $\sigma_d(\A)$; hence it consists of at most countable number of isolated eigenvalues with finite multiplicity.\\
\indent (5) Observe that in view of \eqref{17022016-1449} $\bf F$ is well-defined and continuous as an operator $L^2\to L^2$ since
\begin{equation}\label{w2}
\|{\F}(u)\|_{L^2}\leq \|m\|_{L^2},\;\; u\in L^2,
\end{equation}
and, by \eqref{17022016-1450},
\begin{equation}\label{w3}
\|{\F}(u)-{\F}(v)\|_{L^2}\leq
\|(l_0+l_\infty)|u-v|\|_{L^2}\leq \|l_0\|_{L^p}\|u-v\|_{L^q}+\|l_\infty\|_{L^\infty}\|u-v\|_{L^2}
\leq L \|u-v\|_{H^1},
\end{equation}
for $u,v\in H^1$, with an appropriately chosen Lipschitz constant $L$. Clearly, if $u\in H^1$, then $u\in L^2\cap L^q$ and $\max\{\|u\|_{L^2},\|u\|_{L^q}\}\leq \c\|u\|_{H^1}$ (\footnote{Here and below by $\c$ we denote an appropriate constant for which the given inequality holds; therefore $\c$ may vary from one inequality to another.}). Hence $\F$ is Lipschitz continuous as a map $H^1\to L^2$.\\
\indent (6) By \cite[Theorem 3.3.3]{Henry} (comp. \cite[Chapter 3]{Dlotko}), the sectoriality of ${\A}$, conditions \eqref{w2} and \eqref{w3} imply that for each $\bar u\in H^1$ and $\lambda\in\R$ there is a unique global solution $u$ of
\begin{align}\label{w1}
&\dot u= - {\A}u+\lambda u + {\F}(u),\;\; t>0,\;\lambda\in\R,\;u\in H^1,
\end{align}
i.e. a continuous function $u=u(\cdot;\bar u,\lambda):[0,+\infty)\to H^1$ such that $u\in C((0,+\infty),H^2)\cap C^1((0,+\infty),L^2)$, $u(0)=\bar u$ and \eqref{w1} holds for all $t>0$. \hfill $\square$}\end{remark}

\begin{lemma}\label{zw} The operator ${\V}_0$ is relatively $({\A}_0+{\V}_\infty)$-compact, i.e. $D({\A}_0+{\V}_\infty)\subset D({\V}_0)$ and ${\V}_0$ is compact as a map on $D({\A}_0+{\V}_\infty)$ endowed with the graph-norm.
\end{lemma}
\noindent {\em Proof.} In view of Remark \ref{spectra} (2), $D({\A}_0+{\V}_\infty)=H^2(\R^N)\subset L^q=D({\V}_0)$. Assume that a sequence $(u_n)_{n=1}^\infty$ is bounded in the $H^2$ sense, i.e. $\sup\|u_n\|_{H^2}\leq R$ for some $R>0$. Clearly $\sup\|u_n\|_{L^q}\leq\c R$.
Let $v_n:={\V}_0u_n$, $n\geq 1$; we will show that the set $\{v_n\}_{n=1}^\infty$ is precompact in $L^2$. Take an arbitrary $\eps>0$. For any $n, k\geq 1$,
\be\label{pocz2}\int_{\{|x|\geq k\}}v_n^2\,dx\leq \left(\int_{\{|x|\geq k\}}|V_0|^p\,dx\right)^{2/p}\left(\int_{\{|x|\geq k\}}|u_n|^q\,dx\right)^{2/q}\leq \c R^q\left(\int_{\{|x|\geq k\}}|V_0|^p\,dx\right)^{2/p}<\eps^2\ee
provided $k$ is large enough. Take such $k$, let $B:=\{x\in\R^N\mid |x|<k\}$ and $u_n'=u_n|_B$, $n\geq 1$. Then $u_n'\in H^2(B)$, $(u_n')$ is bounded in $H^2(B)$ and, in view of the compactness of the embedding $H^2(B)\subset L^q(B)$, without loss of generality we may assume that $u_n'\to u_0'$ in  $L^q(B)$ as $n\to\infty$. For $n\geq 0$ let
$$w_n=\begin{cases}V_0u_n\;\;&\text{on}\;\;B,\\
0\;\;&\text{on}\;\;\R^N\setminus B.\end{cases}$$
Then $w_n\to w_0$ in $L^2$ and, by \eqref{pocz2}, $\|v_n-w_n\|_{L^2}<\eps$. It follows that $\{v_n\}_{n=1}^\infty$ is precompact.\hfill $\square$
\begin{remark}\label{comp}{\em (1) The above argument shows actually that ${\V}_0$ is relatively $({\A}_0+{\V}_\infty)$-compact if $p\geq 2$ for $N\leq 3$ and $p>N/2$ for $N\geq 3$; comp. \cite[Theorem 8.19]{Schmud}. The restrictions put on $p$ in \eqref{24062015-1047} are necessary to ensure that $H^1\subset L^q$.\\
\indent (2) An argument similar to the one used in the above proof shows that a bounded subset $M\subset H^1$ is relatively compact in $L^2$ provided  for any $\eps>0$ there is $R>0$ such  that
$$\forall\,u\in M\;\;\;\;\; \int_{\{|x|\geq R\}}|u(x)|^2 \,dx < \eps.\eqno\square$$
}\end{remark}

%\\
% {\color{red}\indent (5) By \eqref{w2} (or its counterpart for $\bold G$) and the Gronwall inequality, given $R>0$, $a,b\in\R$, $a\leq b$, there is $t_0>0$
% such that if $\|\bar u \|_{H^1}\leq R$, then $\|u(t;\bar u,\lambda)\|_{H^1}\leq 2R$ for all $t\in [0,t_0]$ and $\lambda\in [a,b]$.

In view of Remark \ref{spectra} (6), for any $\lambda\in\R$, we are in a position to define $\Phi^\lambda:[0,\infty)\times H^1\to H^1$ by putting
\be\label{wzor semiflow}\Phi^\lambda_t(\bar u):=u(t;\bar u,\lambda),\;\;\bar u\in H^1,\;t\geq 0.\ee
It is immediate to see that $\Phi^\lambda$ is a semiflow on $H^1$. By envoking \cite[Prop. 2.3]{Prizzi-FM} (comp. \cite[Theorem 3.2.1]{Dlotko}, \cite[Prop. 4.3]{Cwiszewski-RL}) we get
the following continuity result.
\begin{proposition}\label{12122016-1244} Given sequences $(\bar u_n)$ in $H^1$ and $\lambda_n\to\lambda$ in $\R$,\\
\indent {\em (i)} if\,  $\bar u_n\to \bar u$ in $H^1$, then $\Phi^{\lambda_n}_t(\bar u_n)\to \Phi^\lambda_t(\bar u)$ uniformly with respect to $t$ in compact subsets of $\R$;
as a consequence the family $\{\Phi^\lambda\}_{\lambda\in \R}$ is continuous;\\
\indent {\em (ii)} if\,  $T>0$, $R>0$, $\|\Phi^{\lambda_n}_t(\bar u_n)\|_{H^1}\leq R$ for all  $t\in [0,T]$ and $\bar u_n\to \bar u$ in $L^2$, then $\Phi^{\lambda_n}_t(\bar u_n)\to \Phi^\lambda_t(\bar u)$ uniformly with respect to $t$ in compact subsets of $(0,T]$. \hfill $\square$
\end{proposition}

Recall the standing assumptions and, as in Theorem \ref{24062015-1056} (i), suppose that
\be\label{eigen}\lambda_0\;\;\text{is an isolated eigenvalue of}\;\;\A\;\;\text{of finite multiplicity and let }\;\;0<\delta<\dist(\lambda_0,\sigma(\A)\setminus\{\lambda_0\}).\ee

Let $X_0:=\Ker ({\A} -\lma_0 {\I})$, $X_\pm$ be the closed subspaces of $L^2$ corresponding to $\sigma(\A)\cap (-\infty,\lambda_0)$, $\sigma(\A)\cap (\lambda_0,+\infty)$, respectively; let $X:=X_-\oplus X_+$ ($\oplus$ stands for the orthogonal sum). It is clear that $X_0$, $X_\pm$ are $\A$-invariant, $L^2=X_0\oplus X$, $\dim X_0,\dim X_-<\infty$ and $X_0, X_-\subset H^2$ since these spaces are spanned by a finite number of  eigenfunctions. Let $\Q_\pm:L^2\to L^2$ be the orthogonal projections onto $X_\pm$, $\Q:=\Q_-+\Q_+$ and $\P:=\I-\Q$. Observe that $\P,\Q_-\in {\mathcal L}(L^2,H^2)$,
$\Q_+(H^2)\subset H^2\cap X_+$ and $\Q_+|_{H^1}\in {\mathcal L}(H^1,H^1)$, i.e.
\be\label{szac q plus}\|{\Q}|_{H^1}\|_{{\mathcal L}(H^1,H^1)}<\infty.\ee
If $|\lambda-\lambda_0|\leq\delta$, then $\lambda\not\in \sigma(\A|_X)$. Hence $(\A-\lambda\I)|_X$ is inveritble and the map
\be\label{odwotne}[\lambda_0-\delta,\lambda_0+\delta]\times X\ni (\lambda,w)\mapsto [(\A-\lambda\I)|_X]^{-1}w\in X\cap H^2\ee
is continuous and $\|[(\A-\lambda\I)|_X]^{-1}w\|_{H^2}\leq \c \|w\|_{L^2}.$

\begin{lemma}\label{odd} The map
$$[\lambda_0-\delta,\lambda_0+\delta]\times L^2\ni (\lambda,u)\mapsto \G(\lambda,u):=\F(\P u+[(\A-\lambda\I)|_X]^{-1}\Q u)\in L^2$$
is completely continuous.
\end{lemma}
\begin{proof} The continuity of $\G$ is evident. Let sequence $(u_n)$ in $L^2$ and $(\lambda_n)$ in $[\lambda_0-\delta,\lambda_0+\delta]$ be bounded. Let $v_n=\P u_n$, $w_n:=\Q u_n$, $\tilde w_n:=[(\A-\lambda_n\I)|_X]^{-1}w_n$ and $z_n:=\G(\lma_n, u_n)$, $n\geq 1$.  Without loss of generality we may assume that $v_n\to v_0\in X_0$.
Take an arbitrary $\eps>0$. In view of \eqref{17022016-1449} there is $R>0$ such that for all $n\geq 1$
\be\label{odd1}\int_{\{|x|\geq R\}}z_n^2\,dx\leq \int_{\{|x|\geq R\}}m^2\,dx<\eps^2.\ee
Let $B=\{x\in\R^N\mid |x|<R\}$, $v_n':=v_n|_B$, $\tilde w_n':=\tilde w_n|B$, $n\geq 1$. Then $v_n'\to v_0':=v_0|_B$; the sequence $(\tilde w_n')$ is bounded in $H^2(B)$ and, thus, we may assume that $\tilde w_n'\to \tilde w_0'\in L^2(B)$ as $n\to\infty$. For $n\geq 0$ let
$$z_n'=\begin{cases}f(x,v_n'(x)+\tilde w_n'(x))\;\;&\text{on}\;\;B,\\
0\;\;&\text{on}\;\;\R^N\setminus B.\end{cases}$$
Then $z_n'\to z_0'$ in $L^2$ and, in view of \eqref{odd1},  $\|z_n-z_n'\|_{L^2}<\eps$. This implies that $\{z_n\}$ is precompact.
\end{proof}

Now, in the context of Theorem \ref{24062015-1056} (ii) we suppose that
\begin{gather}\label{alfa-infty}\lambda_0\in\sigma(\A) \;\; \text{ and }\;\; \lambda_0<\alpha_\infty.\end{gather}
In view of Remark \ref{spectra} (4), $\lambda_0$ is an isolated eigenvalue of finite multiplicity. Take $\delta>0$ such that
\begin{align}\label{spectra 2}&0<\delta<\min\{\alpha_\infty-\lambda_0,\dist(\lambda_0,\sigma({\A})\setminus\{\lambda_0\}\}.\end{align}

\begin{lemma}\label{12102017-1634} {\em (comp. \cite[Proposition 2.2]{Prizzi}, \cite{Cwiszewski-RL})} Let $R>0$ and let $\delta>0$ be given as in \eqref{spectra 2}. There is $\alpha>0$ and a sequence $(\alpha_n)$ with $\alpha_n\searrow 0$ such that if $u:[t_0,t_1] \to H^1$ is a solution of the semiflow $\Phi^\lambda$ corresponding to \eqref{w1} for some $\lma\in [\lma_0-\delta, \lma_0 +\delta]$ such that $\|\Q u(t)\|_{H^1}\leq R$ for all $t\in [t_0,t_1]$, then there is $n_0\geq 1$ such that
\be\label{admis 1}
\forall\,n\geq n_0\quad\quad
\int_{\R^N \setminus B(0,n)} |\Q u(t_1)|^2\, dx \leq
e^{-2\alpha(t_1-t_0)} \|u(t_0) \|_{L^2}^{2} + \alpha_n.
\ee
\end{lemma}
\noindent {\em Proof.} Since $u$ is a solution of $\Phi^\lambda$, we have $u(t+t_0)=\Phi^\lambda_t(u(t_0))$ for $t\in [0,t_1-t_0]$, i.e., in the case of \eqref{w1},
$$
\dot u (t)= -{\mathbf{A}}u(t) + \lma u(t) + {\F}(u(t))\;\; \text{for}\;\; t\in (t_0,t_1].
$$	
For  $w:=\Q u$ and $t\in (t_0, t_1]$ we have
$$\dot w(t)=-{\A}w(t)+\lambda w(t)+\Q{\F}(u(t)).$$
\indent Let $\phi:[0,+\infty)\to [0,1]$ be a smooth
function such that $\phi(s)=0$ if $s\in [0,1/2]$  and $\phi(s)=1$ if $s\geq 1$.
Putting
$$\phi_n(x):=\phi(|x|^2/n^2),\;\; x\in\R^N,$$
we get, for $t\in (t_0, t_1]$ and $n\geq 1$,
\begin{align*}
\frac{1}{2}\frac{d}{d t}\la w(t), \phi_n w(t)\ra_{L^2}&=
	\la\phi_n w(t),\dot w(t)\ra_{L^2}=\la\phi_nw(t),-(\mathbf{A}_0+\mathbf{V}_0+\mathbf{V}_\infty-\lma {\I})
w(t)+\Q\mathbf{F}(u(t))\ra_{L^2}=\\
&=I_1(t) + I_2(t) + I_3(t),
\end{align*}
where
\begin{align*}
I_1(t)&=\;\la\phi_n w(t), - {\A}_0w(t)\ra_{L^2}=-\la\nabla(\phi_nw(t)),\nabla w(t)\ra_{L^2}=
-\int_{\R^N}\phi_n(x)|\nabla w(t)|^2\,dx+\\
&-\frac{2}{n^2}\int_{\{\frac{\sqrt{2}}{2}n\leq |x|\leq n\}} \phi'(|x|^2/n^2)\la w(t)x,\nabla w(t)\ra_{\R^N}\,dx\leq  \frac{2L_\phi}{n}\|w(t)\|_{L^2}\|\|w(t)\|_{H^1}\leq\frac{2L_\phi R^2}{n},
\end{align*}
with
\be\label{Lfi}L_\phi:= \sup_{s\in[0,+\infty)}|\phi'(s)|;\ee
note that $L_\phi < \infty$.\\
\indent In order to estimate the second term  $I_2(t)$, take $0<\eta\leq \frac{1}{2}(\alpha_\infty-\lambda_0-\delta)$. By definition of $\alpha_\infty$ (see \eqref{def alfa}), there is a positive integer $n_0$ such that $V_\infty(x)>\alpha_\infty-\eta$ for a.a. $|x|\geq \sqrt{2}n_0/2$. For $n\geq n_0$ we have
\begin{align*}
	I_2(t)&=\la\phi_nw(t),-({\V}_0+{\V}_\infty-\lambda {\I})w(t)\ra_{L^2}=
	-\la\phi_nw(t),({\V}_\infty-\lambda {\I})w(t)\ra_{L^2}-\la\phi_nw(t),V_0w(t)\ra_{L^2}=\\
&=-\int_{\{\frac{\sqrt{2}}{2}n\leq |x|\leq n\}}\phi_n(x)(V_\infty(x)-\lambda)|w(t)|^2\,dx-\int_{\R^N}\phi_n(x)V_0(x)|w(t)|^2\,dx\leq\\
&\leq -\alpha\la\phi_nw(t),w(t)\ra_{L^2}+\c\|w(t)\|_{H^1}^2
\left(\int_{\{|x|\geq\frac{\sqrt{2}}{2}n\}}|V_0(x)|^p\,dx\right)^{1/p},
\end{align*}
where $\alpha:=\alpha_\infty-\lambda_0-\delta-\eta>0$; the last estimate follows in view of the H\"older inequality since $\|w(t)\|_{L^{2p/p-1}}\leq\c\|u(t)\|_{H^1}$. Finally for all $n\geq 1$
\be\label{oszac}
\begin{split}
I_3(t)&=\la\phi_nw(t),{\Q} {\F}(u(t))\ra_{L^2}\leq
\|w(t)\|_{L^2} \left(\| \phi_n {\F}(u(t))\|_{L^2} + \|\phi_n {\P} {\F}(u(t))\|_{L^2}\right)\leq \\
& \leq R\left(\left(\int_{\{|x|>\frac{\sqrt{2}}{2}n\}} |m(x)|^2 dx \right)^{1/2}
+\ \kappa_n \right),
\end{split}
\ee
where $\kappa_n:=
\sup \left\{  \left( \int_{\{|x|>\frac{\sqrt{2}}{2}n\}} |z(x)|^2 dx \right)^{1/2}
\left|\right. z\in {\P}\left(B\left( 0 , \|m\|_{L^2} \right) \right)
\right\}$ for $n\geq 1$.
Since ${\P}\left(B\left(0,\|m\|_{L^2}\right)\right)$ is relatively compact (as a bounded subset of the finite dimensional space) with respect to the $L^2$ topology, in view of the Kolmogorov-Riesz compactness criterion (see e.g. \cite[Theorem 5]{Hanche}), we see that $\kappa_n \to 0^+$ as $n\to \infty$.\\
\indent Combining these estimates we get that for any $n\geq n_0$
$$
\frac{d}{d t}\la w(t), \phi_n w(t)\ra_{L^2} \leq -2\alpha\la w(t), \phi_n
w(t)\ra_{L^2} + 2\tilde \alpha_n,
$$
where
$$
\tilde \alpha_n:= \frac{2 R^2 L_\phi}{n} +  \c R^2 \bigg(\int_{\big\{|x|\geq \frac{\sqrt{2}}{2}n\big\}}\!\! |V_0(x)|^p d x\bigg)^{1/p} \!\!\! + R \bigg(\!\!\int_{\big\{|x|\geq \frac{\sqrt{2}}{2}n\big\}}|m(x)|^2 d x\bigg)^{1/2}+R\kappa_n.
$$
Multiplying by $e^{2\alpha (t-t_0)}$ and integrating
over $[t_0,t_1]$ one obtains
$$
e^{2\alpha (t_1-t_0 )}\la w(t_1), \phi_n w(t_1)\ra_{L^2}- \la w(t_0), \phi_n w(t_0)\ra_{L^2}
\leq \frac{e^{2\alpha(t_1-t_0)}-1}{\alpha}\tilde  \alpha_n,
$$
This clearly implies
\begin{align*}
\int_{\R^N\setminus B(0,n)} |w(t_1)|^2 d x \leq \la w(t_1), \phi_n w(t_1)\ra_{L^2} \leq e^{-2\alpha (t_1-t_0)}\|w(t_0)\|^2_{L^2} +
\alpha^{-1} \tilde \alpha_n,
\end{align*}
which finally yields the assertion with $\alpha_n:=\frac{\tilde\alpha_n}{\alpha}$.\hfill $\square$

\begin{proposition}\label{26092017-1533}
Let  $R>0$, $\delta$ be as in Lemma  \ref{12102017-1634} and $M_R$ be the set of $\bar u \in H^1(\R^N)$ such that there exists a solution $u:(-\infty, 0]\to H^1(\R^N)$ of $\Phi^\lambda$ for some $\lma\in [\lma_0-\delta,\lma_0+\delta]$ with $u(0)=\bar u$ and $\|\mathbf{Q}u(t)\|_{H^1} \leq R$ for all $t\leq 0$. Then $\mathbf{Q}M_R$ is relatively compact in $L^2(\R^N)$.
\end{proposition}
\noindent {\em Proof}. We will use Remark \ref{comp} (2). Take $\eps>0$ and $t_0<0=t_1$. In view of Lemma  \ref{12102017-1634} there is $\alpha>0$ and a sequence $\alpha_n \searrow 0^+$ (recall that $\alpha_n$ is independent of the choice of $t_0$) such that, for all $\bar u\in M_{R}$ and $n\geq n_0$,
$$
\int_{\R^N \setminus B(0,n)} |\mathbf{Q} \bar u|^2 d x \leq
e^{2\alpha t_0} \|\mathbf{Q} u(t_0)\|_{L^2}^2 + \alpha_n \leq
e^{2\alpha t_0} R^2 + \alpha_n<\eps,
$$
where $u:(-\infty,0]\to H^1(\R^N)$ is the solution of $\Phi^\lambda$  such that
$u(0)=\bar u$, provided that $e^{2\alpha t_0}R^2<\eps/2$ and $\alpha_n<\eps/2$ for $n\geq n_0$.\hfill $\square$

\begin{remark}\label{admiss-arg} {\em  Conclusions of Lemma \ref{12102017-1634} and Proposition \ref{26092017-1533} stay true if the projection ${\Q}$ is replaced by the identity on $L^2(\R^N)$.\hfill $\square$
}\end{remark}

\begin{corollary}{\em (Comp. \cite{Prizzi-FM})}
\label{admiss}
Any bounded set $M\subset H^1$ is admissible with respect to $\{\Phi^\lambda\}_{\lambda\in [\lambda_0-\delta,\lambda+\delta]}$.
\end{corollary}
\begin{proof} Take sequences $t_m\to \infty$, $(u_m)\in H^1$ and $\lambda_m\to\lambda\in [\lambda_0-\delta,\lambda_0+\delta]$ such that $\Phi^{\lambda_m}_{[0,t_m]}(u_m)\subset M$ and $R>0$ such that $M
\subset D_{H^1}(0,R):=\{u\in H^1\mid \|u\|_{H^1}\leq R\}$.
%In view of Remark \ref{spectra} (6) there is $t_0>0$ such that % $\Phi^{\lambda^n}_t(v)\|_{H^1}\leq 2R$ for $t\in [0,t_0]$ if $\|v\|_{H^1}\leq R$.
With no loss of generality we may assume that $t_m>t_0$ for all $m$.
Then, for all $m$,
$$
\Phi_{t_m}^{\lma_m} (u_m) = \Phi_{t_0}^{\lma_m} (z_m)
$$
where $z_m := \Phi_{t_m-t_0}^{\lma_m}(u_m)$. It follows from Lemma \ref{12102017-1634} that, for all $m, n\in \N$,
$$
\int_{\R^N\setminus B(0,n)} |z_m(x)|^2\, dx \leq e^{-2\alpha (t_m-t_0)}    \|u_m\|_{L^2}+\alpha_n \leq R^2 e^{-2\alpha (t_m-t_0)}+\alpha_n
$$
where $\alpha_n\to 0^+$ as $n\to \infty$. This, in view of Remark \ref{comp} (2), means that the sequence $(z_m)$ is relatively compact in $L^2$. Now, by the weak relative compactness of bounded sets in $H^1$, there exists $z\in H^1$ such that (up to a subsequence), $z_m \rightharpoonup z$ (weakly) in $H^1$
and $z_m \to z$ in $L^2$. Thus, by Proposition \ref{12122016-1244}, $\Phi^{\lambda_m}_{t_m}(u_m)=\Phi^{\lambda_m}_{t_0}(z_n)\to \Phi_{t_0}^{\lambda}(z)$. \end{proof}

\begin{remark}\label{after-admiss} {\em
(1) Observe that if $u:\R\to H^1$ is a full bounded solution of $\Phi^\lma$ for some $\lma\in [\lma_0-\delta, \lma_0 + \delta]$, then the set $u(\R)$ is relatively compact (in $H^1$). Indeed: for any $(t_n) \in \R$ one has $u(t_n)=\Phi_{n}^{\lma} (z_n)$ with $z_n = u(t_n-n)$, $n\in \N$, that are contained in a bounded set; hence, by Corollary \ref{admiss}, $(u(t_n))$ contains a convergent subsequence.\\
\indent (2) Let the functional $J_\lma:H^1\to \R$, $\lma\in [\lma_0-\delta, \lma_0+\delta]$, be given by
$$
J_\lambda(v):= \frac{1}{2} \int_{\R^N} (|\nabla u|^2 + V(x)|u|^2 - \lma |u|^2)\, dx - \int_{\R^N}  F (x,u)\, dx
$$	
where $F(x,s):=\int_{0}^{s} f(x,\tau)\,d\tau$. Then, for any solution $u:(t_0,t_1)\to H^1$ of $\Phi^\lambda$, one has
$$
\frac{d}{dt}\left[ J_\lambda(u(t)) \right] = - \|\dot u(t)\|_{L^2}^2 \ \mbox{ for each } \ t\in (t_0,t_1).
$$
This means that $J_\lma$ is a Liapunov-function for $\Phi^\lma$, i.e. it decreases  along solutions of $\Phi^\lambda$. It is also clear that if a solution $u$ is nonconstant, then so is $t\to J(u(t))$. Therefore, if  $u:\R\to H^1$ is a full bounded solution of $\Phi^\lma$, then the limit sets $\alpha(u)$ and $\omega(u)$ consists only of equilibria of $\Phi^\lma$ (see \cite[Prop. 5.3]{rybakowski}).\hfill $\square$
}
\end{remark}

The following Conley index formula, obtained by linearization and  Theorem \ref{11052015-1744}, will be used in the sequel.
\begin{proposition} \label{13122016-1923} {\em (comp. \cite[Theorem 3.3]{Prizzi-FM})}
Under assumptions \eqref{24062015-1047},  \eqref{17022016-1449} and \eqref{17022016-1450}, suppose that $\lambda\not\in\sigma({\A})$
and $\lambda<\alpha_\infty$. Denote by $K(\Phi^\lambda)$ the set of all $\bar u\in H^1$ such that there exists a bounded solution $u:\R\to H^1$ of $\Phi^\lambda$ such that $u(0)=\bar u$. Then $K(\Phi^\lambda)$ is bounded, isolated invariant with respect to $\Phi^\lambda$, $(\Phi^\lambda,K(\Phi^\lambda))\in {\mathcal I}(H^1)$  and the Conley index
$$
h(\Phi^\lambda, K(\Phi^\lambda)) = \Sigma^{k(\lma)}
$$
where $k(\lambda)$ is the total multiplicity of the negative eigenvalues of ${\A}-\lambda {\I}$, i.e. eigenvalues of $-\Delta+V$ less than $\lambda$.\hfill $\square$
\end{proposition}

\section{Necessary conditions}

Below we provide necessary conditions for bifurcation from infinity and  study
additional properties of bifurcation sequences.

\begin{theorem}\label{s16022016-1655}
If a bifurcation from infinity for \eqref{24062015-1032}
occurs at $\lma_0\not\in\sigma_e({\A})$, i.e., there is a sequence $(u_n,\lambda_n)$ solving  \eqref{24062015-1032} with $\lambda=\lambda_n$, $\|u_n\|_{H^1}\to\infty$, $\lambda_n\to\lambda_0$, then $\lambda_0$ lies in $\sigma_p({\A})$ the point spectrum of $\A$ and $\|{\P}u_n\|_{L^2},\|\nabla{\P}u_n\|_{L^2}\to\infty$ as $n\to\infty$. This implies that
$\|u_n\|_{L^2}, \|\nabla u_n\|_{L^2}\to\infty$, too. Moreover the sequences $(\|\Q u_n\|_{L^2})$ and $(\|\nabla\Q u_n\|_{L^2})$ are bounded.\\
\indent  If, additionally $\lambda_0<\alpha_\infty$, then the sequences $(\|u_n\|_{L^2})$ and $(\|\nabla u_n\|_{L^2})$ have the same growth rates, i.e.,  there  are constants $C_1, C_2>0$ such that, for all large $n$,
\begin{equation}\label{24012018-1201}
C_1 \|u_n\|_{L^2} \leq \|\nabla u_n\|_{L^2} \leq C_2 \|u_n\|_{L^2};
\ee
a similar estimate holds for  $\|{\P}u_n\|_{L^2}$ and $\|\nabla{\P}u_n\|_{L^2}$ with large $n$.
\end{theorem}
\noindent {\em Proof}. Let $\rho_n:=\|u_n\|_{H^1}$; we may assume that $\rho_n>0$ for all $n$. Let
$z_n := \varrho_n^{-1} u_n$; then $\|z_n\|_{H^1}=1$ and $\|z_n\|_{L^2}\leq \c$
Suppose to the contrary that $\lambda_0\not\in\sigma_p({\A})$. Since $\lambda_0\not\in\sigma_e({\A})$, this implies that $\lambda_0\in \rho({\A})$, the resolvent set of $\A$. We have
$$
(\mathbf{A}-\lambda_0{\I}) z_n= (\lma_n-\lambda_0) z_n + \rho_{n}^{-1} \mathbf{F}(\rho_n z_n).
$$
Clearly $v_n:=(\lambda_n-\lambda_0)z_n+\rho_n^{-1}\mathbf{F}(\rho_n z_n)\to 0$ as $n\to\infty$ (in $L^2$). Hence $z_n=({\A}-\lambda_0{\I})^{-1}v_n\to 0$ in $H^1$: a contradiction.\\
\indent Since $\lambda_0$ is isolated in $\sigma({\A})$, there is $c>0$ such that for large $n$
we have $\la ({\A}-\lambda_n{\I})v,v\ra_{L^2}\geq c\|v\|_{L^2}^{2}$ for $v\in X_+$ and
$\la ({\A}-\lambda_n{\I})w,w\ra_{L^2}\leq -c\|w\|_{L^2}^{2}$ for $w\in X_-$. This implies that for large $n$
\begin{align*}
c\|{\Q}_\pm u_n\|^2_{L^2}&\leq \pm\la({\A}-\lambda_n{\I}){\Q}_\pm u_n,{\Q}_\pm u_n\ra_{L^2}=\pm\la({\A}-\lambda_n{\I})u_n,{\Q}_\pm u_n\ra_{L^2}=\\
&=\pm\la {\F}(u_n),{\Q}_\pm u_n\ra_{L^2}\leq \|m\|_{L^2}\|{\Q}_\pm u_n\|_{L^2}.\end{align*}
Therefore for large $n$
\be\label{b1}\|\Q u_n\|_{L^2}\leq 2c^{-1}\|m\|_{L^2}.\ee
On the other hand
$$\|\nabla{\Q}u_n\|_{L^2}^2+\la({\V}-\lambda_n{\I}){\Q}u_n,{\Q}u_n\ra_{L^2}=\la({\A}-
\lambda_n{\I})u_n,{\Q}u_n\ra_{L^2} =\la{\F}(u_n),{\Q}u_n\ra_{L^2}.
$$
Hence
$$\|\nabla{\Q}u_n\|^2_{L^2}\leq\|V_\infty-\lambda_n\|_{L^\infty}\|{\Q}u_n\|_{L^2}^2+
\|V_0\|_{L^p}\|{\Q}u_n\|^2_{L^s}+\|m\|_{L^2}\|{\Q}u_n\|_{L^2},$$
where $s:=2p/(p-1)$.  Clearly, $s>2$ and, if $N\geq 3$, one has also $s<2_{N}^*=2N/(N-2)$. In view of the Gagliardo-Nirenberg inequality (see Remark \ref{G-N})
\be\label{b2}\|\nabla{\Q}u_n\|^2_{L^2}\leq\|V_\infty-\lambda_n\|_{L^\infty}\|{\Q}u_n\|_{L^2}^2
+C^2\|V_0\|_{L^p}\|\nabla{\Q}u_n\|_{L^2}^{2\theta}\|{\Q}u_n\|_{L^2}^{2(1-\theta)}+
\|m\|_{L^2}\|{\Q}u_n\|_{L^2}
\ee
for some $C>0$ and $\theta\in (0,1)$. This, together with \eqref{b1}, implies that the sequence $(\|\nabla{\Q}u_n\|_{L^2})$ is bounded.\\
\indent The same argument (replacing ${\Q}$ in \eqref{b2} by the identity $\I$) shows that would $(\|\nabla u_n\|_{L^2})$ be bounded if $(\|u_n\|_{L^2})$ were bounded. Since $\|u_n\|_{H^1}\to\infty$, we deduce therefore  that $\|u_n\|_{L^2}\to\infty$. Now
$\|{\P}u_n\|_{L^2}^2=\|u_n\|_{L^2}^2-\|{\Q}u_n\|^2_{L^2}$, so $\|{\P}u_n\|_{L^2}\to\infty$ in view of \eqref{b1}. This implies that also $\|\nabla{\P}u_n\|_{L^2}\to\infty$ because $\dim X_0<\infty$. Since
$$\|\nabla u_n\|_{L^2}\geq |\|\nabla{\P}u_n\|_{L^2}-\|\nabla{\Q}u_n\|_{L^2}|,$$
we finally infer that $\|\nabla u_n\|_{L^2}\to\infty$.\\
\indent Now assume that $\lambda_0<\alpha_\infty$. Take $\eta>0$ such that $\lambda_0+3\eta<\alpha_\infty$ and $R>0$ such that $V_\infty(x)\geq \alpha_\infty-\eta$ for a. a. $x\in\R^N$ with $|x|>R$. Then for large  $n\geq 1$, $V_\infty(x)-\lambda_n>\eta$ a.e. on $\{x\in\R^N\mid |x|>R\}$.\\
\indent For large $n$ we have
\begin{align*}
\int_{\R^N}|\nabla u_n|^2\,dx+\eta\int_{\R^N}u^2\,dx&\leq\int_{\R^N}|\nabla u|^2\,dx+
\int_{\{|x|>R\}}(V_\infty(x)-\lambda_n)u_n^2\,dx+\eta\int_{\{|x|\leq R\}}u_n^2\,dx=\\
&=\int_{\R^N}(|\nabla u|^2+(V_\infty(x)-\lambda_n)u_n^2)\,dx+\int_{\{|x|\leq R\}}(\eta-V_\infty(x)+\lambda_n)u_n^2\,dx.
\end{align*}
Hence
\be\label{wk}\|\nabla u_n\|_{L^2}^2+\eta\|u_n\|_{L^2}^2\leq-\int_{\R^N}V_0(x)u_n^2\,dx+\int_{\{|x|\leq R\}}(\eta-V_\infty(x)+\lambda_n)u_n^2\,dx+\int_{\R^N}f(x,u_n)u_n\,dx.\ee
Take $\xi>0$ such that $\xi\geq |\eta-V_\infty(x)-\lambda_n|$ for all large $n$ and let $V_1(x)=\xi$ if $|x|\leq R$ and $V_1(x)=0$ otherwise. Then $V_1\in L^p$ and, by \eqref{wk} we have
$$\|\nabla u_n\|_{L^2}^2+\eta\|u_n\|_{L^2}^2\leq \|V_0+V_1\|_{L^p}\|u_n\|^2_{L^s}+\|m\|_{L^2}\|u_n\|_{L^2}$$
and, again in virtue of the Gagliardo-Nirenberg inequlaity, we get that
\be\label{b3}
\|\nabla u_n\|_{L^2}^{2}+\eta\|u_n\|_{L^2}^{2}
\leq C^2 \|V_0+V_1\|_{L^p} \|\nabla u_n\|_{L^2}^{2\theta}\|u_n\|_{L^2}^{2(1-\theta)}
+\|m\|_{L^2} \|u_n\|_{L^2}
\ee
with constants $C>0$ and $\theta\in (0,1)$. For large $n$,
$$
\left(\frac{\|\nabla u_n\|_{L^2}}{\|u_n\|_{L^2} }\right)^2 + \eta
\leq C^2\|V_0+V_1\|_{L^p} \left(\frac{\|\nabla u_n\|_{L^2}}{\|u_n\|_{L^2} }\right)^{2\theta} + 1
$$
and
$$
1+ \eta\left(\frac{\|u_n\|_{L^2} }{\|\nabla u_n\|_{L^2}}\right)^2
\leq C^2\|V_0+V_1\|_{L^p} \left(\frac{\|u_n\|_{L^2}}{\|\nabla u_n\|_{L^2}}\right)^{2(1-\theta)} + \frac{\|u_n\|_{L^2} }{\|\nabla u_n\|_{L^2}},
$$
which gives the existence of $C_1, C_2>0$ satisfying  \eqref{24012018-1201}. A similar argument shows that growth rates of $(\|{\P}u_n\|_{L^2})$ and $\|\nabla{\P}u_n\|_{L^2})$ are the same. \hfill $\square$

\begin{remark}\label{G-N} {\em
The {\em Gagliardo-Nirenberg inequality} (see \cite{Nirenberg} and  \cite{Agueh}) states  that given $1<r<s$ (with $s<2_{N}^*=\frac{2N}{N-2}$ if $N\geq 3$) there are $C>0$ and $\theta\in (0,1)$ such that for any $u\in H^1$
$$
\|u\|_{L^s}\leq C\|\nabla u\|_{L^2}^{\theta} \|u\|_{L^r}^{1-\theta} \ \
\mbox{ for all } \ \ u\in H^1.
\eqno \square$$		}
\end{remark}
Theorem \ref{s16022016-1655}  shows that bifurcating sequences $(u_n)$ are localized around the eigenspace $\Ker({\A}-\lambda_0{\I})$ having mass $\|u_n\|_{L^2}$ and energy of the same growth rate.
It generalizes \cite[Theorem 5.2 (iii)]{Stuart1}, where the case of a simple eigenvalue has been studied.

\section{Sufficient conditions - proof of Theorem \ref{24062015-1056}}
Recall the notation introduced in front of Lemma \ref{odd}. We start with the {\em proof of Theorem \ref{24062015-1056}} (i): assume \eqref{eigen}, let $\dim X_0$ be odd and suppose that there is no asymptotic bifurcation at $\lambda_0$. Taking smaller $\delta>0$ if necessary there is $r>0$ such that for all $\lambda\in [\lambda_0-\delta,\lambda_0+\delta]$ if $w\in H^2$ and $(\A-\lambda\I)w=\F(w)$, then $\|w\|_{H^1}\leq r$.\\
\indent Observe that $w\in H^2$, solves \eqref{pr1} with some $\lambda\in [\lambda_0-\delta,\lambda_0+\delta]$, i.e. $(\A-\lambda\I)w=\F(w)$,
if and only if $u:=\P w+(\A-\lambda\I)\Q w\in L^2$ solves
\be\label{odd2}u=\K(\lambda,u):=(1+\lambda-\lambda_0)\P u+\F(\P u+[(\A-\lambda\I)|_X]^{-1}\Q u)=
(1+\lambda-\lambda_0)\P u+\G(\lambda,u);\ee
see Lemma \ref{odd}. Here the nonlinearity $\K:[\lambda_0-\delta,\lambda_0+\delta]\times L^2\to L^2$ is continuous and, in view of Lemma \ref{odd}, completely continuous. Moreover \eqref{odd2} has no solutions if $|\lambda-\lambda_0|\leq\delta$ and $\|u\|_{L^2}$ is sufficiently large. Indeed if $u\in L^2$ solves \eqref{odd2}, where $|\lambda-\lambda_0|\leq\delta$, then $w:=\P u+[(\A-\lambda\I)|_X]^{-1}\Q u$ solves \eqref{pr1}, i.e.,
$\|\P u\|_{L^2}=\|\P w\|_{L^2}\leq\|w\|_{H^1}\leq r$. Hence $\|u\|_{L^2}\leq (1+\delta)r+\|m\|_{L^2}:= R_0$.
Therefore the Leray-Schauder fixed-point index $\ind(\K(\lambda,\cdot),B)$, where $B$ is the ball around 0 of radius $R>\max\{R_0,\delta^{-1}\|m\|_{L^2}\}$ in $L^2$, is well-defined and {\em independent} of $\lambda\in [\lambda_0-\delta,\lambda_0+\delta]$. It is immediate to see that if $\lambda=\lambda_0\pm\delta$, then $u\neq (1+\lambda-\lambda_0)\P u+t\G(\lambda,u)$ for $u\not\in B$ and $t\in [0,1]$. Hence, in view of the homotopy invariance and the restriction property of the index, for $\lambda=\lambda_0\pm\delta$
$$\ind(\K(\lambda,\cdot),B)=\ind((1\pm\delta)\P,B)=\ind((1\pm\delta)\I,B\cap X_0).$$
However
$$\ind((1-\delta)\I,B\cap X_0)=1, \;\ind((1+\delta)\I,B\cap X_0)=(-1)^{\dim X_0}=-1.$$
This is a contradiction.\hfill $\square$
\begin{remark} {\em The standard use of the Kuratowski-Whyburn lemma makes it easy to get a slightly better result in the context of Theorem \ref{24062015-1056} (i). Namely it appears that there exists a closed connected set $\Gamma\subset H^2\times\R$ of solutions to \eqref{24062015-1032} which contains a sequence $(u_n,\lambda_n)$ such that $\|u_n\|_{H^2}\to\infty$, $\lambda_n\to\lambda_0$.}\end{remark}
\indent Now we shall pass to the proof of Theorem \ref{24062015-1056} (ii).  We start with the geometric interpretation of the resonance assumptions in  spirit of \cite{Cesari} and  \cite{Kokocki}.
\begin{lemma}\label{25092017-1854} Assume that $M\subset X$. If either\\
\indent {\em (i)} condition $(LL)_\pm$ holds and $M$ is bounded in $L^2$, or\\
\indent {\em (ii)} condition $(SR)_\pm$ holds and $M$ relatively compact in $L^2$,\\
then there exist $R_0>0$ and $\alpha>0$ such that for all  $\bar v\in X_0$ with $\|\bar v\|_{L^2}\geq R_0$ and $\bar w\in M$
\begin{equation}\label{25092017-2303}
\pm\la\bar v, \mathbf{F}(\bar v+\bar w)\ra_{L^2} >\alpha.
\end{equation}
\end{lemma}
\begin{proof}
We carry out the proof for $(LL)_+$ and $(SR)_+$; other cases may be treated analogously. Suppose to the contrary that for any $n\in\N$ there are $\bar v_n\in X_0$ and $\bar w_n\in M$ such that $\|\bar v_n\|_{L^2}\geq n$ and
\begin{equation}\label{31102017-1613}
\la\bar v_n, {\bf F}(\bar v_n+\bar w_n)\ra_{L^2} \leq n^{-1}.\end{equation}
Let $\rho_n:= \|\bar v_n\|_{L^2}$ and $\bar z_n:= \rho_n^{-1}\bar v_n$, $n\in\N$. Since $\dim X_0<\infty$,  we may assume that $\|\bar z_n-\bar z_0\|_{L^2}\to 0$ as $n\to\infty$, where $\bar z_0\in X_0$ and $\|\bar z_0\|_{L^2 (\R^N)}=1$. Therefore we may assume that $\bar z_n(x)\to \bar z_0(x)$ for a.a. $x\in\R^N$ and there is $\kappa\in L^2$ such that $|\bar z_n|\leq\kappa$ a.e.  In view of the so-called {\em unique continuation property} (see e.g.  \cite[Proposition 3, Remark 2]{Gossez}), $\bar z_0\neq 0$ a.e. Hence the set $\R^N\setminus (A_+\cup A_-)$, where $A_\pm:=\{x\in\R^N\mid \pm \bar z_0>0\}$, is of measure zero.\\
\indent Dividing \eqref{31102017-1613} by $\rho_n$ we get
$$
n^{-2}\geq \rho_n^{-1}/n \geq \la\bar z_n, \mathbf{F}( \rho_n \bar z_n+\bar w_n)\ra_{L^2} = \int_{\R^N} \bar z_n(x) f(x, \rho_n  \bar z_n(x)+\bar w_n(x))\, d x.
$$
Assume (i); then $\rho_n^{-1}\bar w_n\to 0$ in $L^2$ since  $M$ is bounded. We may assume without loss of generality that $\rho_n^{-1}\bar w_n(x)\to 0$ for a.a. $x\in\R^N$.  Hence $\bar z_n+\rho_n^{-1}\bar w_n\to \bar z_0$ a.e. This implies that $\rho_n\bar z_n+\bar w_n\to\pm\infty$ for a.a. $x\in A_\pm$.
Using \eqref{17022016-1449}  we are in a position to use the Fatou lemma to get
\begin{align*} 0\geq \liminf_{n\to \infty} \int_{\R^N} \bar z_n f(x, \rho_n \bar z_n+\bar w_n)\,dx&\geq
\int_{\R^N}\liminf_{n\to\infty}\bar z_n f(x,\rho_n\bar z_n+\bar w_n)\,dx\geq\\&\geq \int_{A_+} \check{f}_+\bar z_0\, d x+ \int_{A_-} \hat{f}_-\bar z_0\, d x >0,
\end{align*}
in view of Remark \ref{LL}; this is a contradiction.\\
\indent Assume (ii). Since now $M$ is $L^2$-precompact, we may assume that $\bar w_n\to \bar w_0\in L^2(\R^N)$, $\bar w_n(x)\to \bar w_0(x)$ for a.e. $x\in\R^N$ and there is $\gamma\in L^2(\R^N)$ such that $|\bar w_n|\leq \gamma$  a.e. on $\R^N$ for all $n\in\N$.\\
\indent Clearly
$\la\bar v_n, \mathbf{F}(\bar v_n+\bar w_n)\ra_{L^2}  =  \la\bar v_n + \bar w_n, \mathbf{F}(\bar v_n+\bar w_n)\ra_{L^2} - \la\bar w_n, \mathbf{F}(\bar v_n+\bar w_n)\ra_{L^2}$.
In view of $(SR)_+$, $\lim_{s\to\pm\infty}f(x,s)=0$ for a.a. $x\in\R^N$. Hence, again by \eqref{17022016-1449} and the Lebesgue dominated convergence theorem we have
$$
\la\bar w_n, \mathbf{F}(\bar v_n+\bar w_n)\ra_{L^2}=\int_{\R^N} \bar w_n(x) f(x, \rho_n \bar z_n(x)+\bar w_n(x))\, d x \to 0, \mbox{ as } n\to +\infty,
$$
and, in view of  \eqref{31102017-1613}, arguing as before
\begin{gather*}
0\geq \liminf_{n\to \infty}\la\bar v_n + \bar w_n, \mathbf{F}(\bar v_n+\bar w_n)\ra_{L^2}
=\liminf_{n\to \infty}\int_{\R^N} (\rho_n \bar z_n(x)+\bar w_n(x)) f(x, \rho_n z_n(x)+\bar w_n(x))\,dx>0  \\
\geq \int_{A_+} \check{k}_+ (x)d x + \int_{A_-} \check{k}_-  (x)\,d x>0,
\end{gather*}
we reach a contradiction.
\end{proof}

The set of stationary points of the semiflow $\Phi^\lambda$ related to \eqref{w1}, where $|\lambda-\lambda_0|\leq\delta$ and $\delta$ is given by \eqref{spectra 2} will be denoted  by ${\mathcal E}_\lma$ and let
$${\mathcal E}:= \bigcup_{\lma\in [\lma_0-\delta,\lma_0+\delta]} {\mathcal E}_\lma.$$
\begin{lemma}\label{18022016-0914}
Suppose that  there is $r>0$ such that ${\mathcal E}\subset B_{H^1} (0,r)$ {\em (\footnote{$B_X(x,r)$ (resp. $D_X(x,r)$) stand for the open (resp. closed) ball at $x$ of radius $r>$ in the Banach space $X$.})}. Then there exists $R_\infty=R_\infty(r) > 0$  such that, for any bounded solution $u:\R \to H^1$ of $\Phi^\lambda$ with $\lambda\in [\lma_0-\delta,\lma_0+\delta]$, one has
$$
\sup_{t\in\R} \|\mathbf{Q}\, u(t)\|_{H^1} < R_\infty.
$$
\end{lemma}
\begin{proof} Since $\delta<\dist(\lambda_0,\sigma({\A})\setminus\{\lambda_0\})$, there is $c>0$ such that $\sigma (({\mathbf A}-\lma {\I})_{|X_-}) \subset (-\infty, -c)$ and $\sigma ({\mathbf{A}}-\lma {\I})_{|X_+}) \subset (c,+\infty)$ whenever $|\lambda-\lambda_0|\leq\delta$. \\
\indent Fix $\lambda\in [\lambda_0-\delta,\lambda_0+\delta]$ and let ${\B}_\pm:=({\A}-\lambda {\I})_{|X_\pm}$. Clearly ${\B}_+$ is sectorial and positive. By \cite{Dlotko} (see Corollary 1.3.5 and comp. Corollary 1.3.4)  the domain $D({\B}_+^{1/2})=D({\A}_0^{1/2})=H^1\cap X_+$; thus, in view of \cite[Proposition 1.3.6]{Dlotko}, there is $K>0$ independent of $\lma\in [\lma_0-\delta, \lma_0+\delta]$ such that for all $\tau>0$
\be\label{I-sze}\forall\,v\in X_+\;\;\;\;\ \|e^{-\tau {\B_+}}v\|_{H^1}=\|{\B}_+^{1/2}e^{-\tau{\B_+}}v\|_{L^2}\leq K\frac{e^{-c\tau}}{\tau^{1/2}}\|v\|_{L^2},\;\; \tau>0,\ee
and
\be\label{I-sze'}\forall\,v\in H^1\cap X_+\;\;\;\; \|e^{-\tau {\B_+}}v\|_{H^1}\leq Ke^{-c\tau}\|v\|_{H^1}.\ee
where $\{e^{-\tau {\B}_+}\}_{\tau\geq 0}$ denotes the semigroup generated by $-{\B}_+$.\\
\indent  The semigroup $\{e^{-\tau {\B}_-}\}_{\tau\geq 0}$ generated by ${\B}_-$ is uniformly continuous, i.e. it extends to a strongly continuous group and there is $K'> 0$
 independent of $\lma\in [\lma_0-\delta, \lma_0+\delta]$  such that
\be\label{II-gie}\forall\,v\in X_-\;\; \|e^{-\tau {\B}_-}v\|_{L^2}\geq \frac{1}{K'}e^{c\tau}\|v\|_{L^2},\;\; \tau\geq 0 \ee
since $\sigma({\B}_-)<-c$.\\
\indent Now take a solution $u:\R\to H^1$ of the semiflow $\Phi^\lambda$ corresponding to \eqref{w1}. It is well-known that $u$  is a mild solution (see \cite{Henry}), i.e. the so-called Duhamel formula holds
\begin{equation}\label{27022016-2116}
u(t) = e^{-(t-s)(
{\A}-\lma {\I})} u(s)+\int_{s}^{t} e^{-(t-\tau)({\mathbf{A}}-\lma {\I})} \mathbf{F}(u(\tau))\, d\tau \mbox{ for all } s,t\in\R, \, t>s,
\end{equation}	where $\{e^{-\tau({\A}-\lambda {\I})}\}_{\tau\geq 0}$ denotes the analytic semigroup generated by $-({\A}-\lambda {\I})$.\\
\indent Since, due to Remark \ref{after-admiss} (2), $\alpha(u)\subset {\mathcal E}_{\lma}$, there exists $t_u<0$ such that $\| u(\tau ) \|_{H^1}<2r$ for all  $\tau \leq t_u$. Thus,  by \eqref{I-sze'}, \eqref{szac q plus} and \eqref{I-sze}, for $t\geq t_u$
\begin{gather*}
\| {\Q}_+u(t)\|_{H^1} \leq \|e^{-(t-t_u) {\B}_+} {\Q}_+ u(t_u) \|_{H^1} + \int_{t_u}^{t} \| e^{-(t-\tau){\B}_+} {\Q}_+{\F}(u(\tau))\|_{H^1}\, d\tau\leq\\
K\left(\|{\Q}_+\|_{{\mathcal L}(H^1,H^1)}e^{-c (t-t_u)}2r + \int_{t_u}^{t} (t-\tau)^{-1/2} e^{-c (t-\tau)}\|{\Q}_+{\F}(u(\tau))\|_{L^2}\,d \tau\right).\end{gather*}
In view of \eqref{17022016-1449}
\be\label{est 1}\|{\Q}_+{\F}(u(\tau))\|_{L^2}\leq \|m\|_{L^2},\;\; t_u\leq\tau\leq t;\ee
thus
$$\| {\Q}_+u(t)\|_{H^1}\leq K\left(\|{\Q}_+\|_{{\mathcal L}(H^1,H^1)} 2r + \|m\|_{L^2} \int_{0}^{+\infty} s^{-1/2} e^{-c s}\,ds\right) =: R'_{1,\infty}.$$
This means that $\|{\Q}_+ u(t)\|_{H^1} \leq R_{1,\infty}=\max\{2r,R'_{1,\infty}\}$ for all $t\in\R$.\\
\indent Since,  due to Remark \ref{after-admiss} (2), $\omega(u)\subset \mathcal{E}_\lma$ we can take $s_u\in\R$ such that $\| u(\tau)\|_{H^1} \leq 2r$, for all  $\tau \geq s_u$, and observe that, in view of \eqref{27022016-2116}, we have for each $t< s_u$
$$
{\Q}_- u(s_u) = e^{-(s_u-t){\B}_-}{\Q}_- u(t) + \int_{t}^{s_u} e^{-(s_u-\tau){\B}_-} {\Q}_- {\F}(u(\tau))\, d \tau.
$$
Hence, using \eqref{II-gie}, we get
$$
\| {\Q}_- u(t)\|_{L^2}\leq K'\left(e^{c (t-s_u)} \|{\Q}_- u(s_u)\|_{L^2}+\int_{t}^{s_u}e^{c(t-\tau)} \|{\Q}_- {\F} (u(\tau))\|_{L^2}\,d \tau\right).$$
Again in view of  \eqref{17022016-1449}
\be\label{est 2}\|{\Q}_-{\F}(u(\tau))\|_{L^2}\leq \|m\|_{L^2},\;\; t\leq\tau\leq s_u.\ee
Therefore
$$\| {\Q}_- u(t)\|_{L^2}\leq K'(2r +\|m\|_{L^2} c^{-1})=:R'_{2,\infty}$$
and thus $\|{\Q}_- u(t)\|_{L^2} \leq \tilde R_{2,\infty}:=\max\{2r,R'_{2,\infty}\}$ for all $t\in\R$. Since $X_-$ is finite dimensional, there is a constant $R_{2,\infty}>0$ such $\|{\Q}_- u(t)\|_{H^1}\leq R_{2,\infty}$ for all $t\in\R$.\\
\end{proof}

\begin{lemma}\label{25092017-2310}
If $u:[t_0,t_1]\to H^{1}(\R^N)$ is a solution of $\Phi^\lambda$ for some $\lma\in\R$, then
$$
\frac{1}{2}\frac{d}{d t} \| \mathbf{P}u(t)\|_{L^2}^{2} = (\lma-\lma_0) \|\mathbf{P}u(t)\|_{L^2}^{2} + \la\mathbf{P}u(t), \mathbf{F}(u(t)\ra_{L^2}, t\in (t_0,t_1).
$$
when $u$ solves \eqref {w1}.
\end{lemma}
\begin{proof} The symmetry of ${\mathbf{A}}$ implies that $X_0$ is orthogonal to to the range $R({\mathbf{A}}-\lma_0 {\I})$ in $L^2$. Hence
\begin{align*}
\frac{1}{2}\frac{d}{d t}\| \mathbf{P}u(t)\|_{L^2}^{2} &= \la\mathbf{P}u(t), \dot u(t)\ra_{L^2}
= \la\mathbf{P}u(t), -({\mathbf{A}}-\lma_0 {\I}) u(t) +(\lma-\lma_0) u(t)+\mathbf{F}(u(t))\ra_{L^2}\\
&=(\lma-\lma_0) \|\mathbf{P}u(t)\|_{L^2}^{2} + \la\mathbf{P}u(t), \mathbf{F}(u(t)\ra_{L^2}
\end{align*}
for all $t\in (t_0,t_1)$.
\end{proof}

\begin{proof}[Proof of  Theorem \ref{24062015-1056} (ii)] Assume \eqref{alfa-infty} and suppose to the contrary that $\lma_0$ is not a point of bifurcation from infinity. Thus there are $r>0$ and $\delta>0$ satisfying condition \eqref{spectra 2} such that
\begin{equation}\label{08072016-0939}
{\mathcal E} \subset B_{H^1} (0,r).
\end{equation}
By Proposition \ref{13122016-1923}, there is $R>0$ such that
the  $K(\Phi^{\lma_0\pm \delta})\subset B_{H^1}(0,R)$.
By Lemma  \ref{18022016-0914} one has $R_\infty=R_\infty(r) \geq R$ such that, for any bounded solution of $u:\R\to H^1$ of $\Phi^\lambda$, $|\lambda-\lambda_0|\leq\delta$, one has
\begin{equation}\label{02112017-2001}
\sup_{t\in \R} \|\mathbf{Q} u (t)\|_{H^1} < R_\infty.
\end{equation}
\indent Let $M_{R_\infty}$ be the set of all $\bar u\in H^1$
such that there exists a solution $u:(-\infty, 0]\to H^1$ of $\Phi^\lma$, $|\lambda-\lambda_0|\leq \delta$, with $u(0)=\bar u$ and $\|\mathbf{Q} u(t)\|_{H^1} \leq R_\infty$ for all $t\leq 0$. In view of Proposition \ref{26092017-1533}, the set $M:=\mathbf{Q}M_{R_\infty}\subset X$ is relatively  compact in $L^2$.
By Lemma \ref{25092017-1854} there are  $R_0 \geq R_\infty$ and $\alpha>0$ such that
for all $\bar v\in X_0\setminus B_{L^2}(0,R_0)$ and $\bar w\in M$
\begin{equation}\label{27092017-2022}
\la \bar v, {\F}(\bar v+\bar w)\ra_{L^2} >\alpha
\end{equation}
if $(LL)_+$ or $(SR)_+$ is satisfied, or
\begin{equation}\label{27092017-20220}
\la \bar v, {\F}(\bar v+\bar w)\ra_{L^2}<-\alpha
\end{equation}
if $(LL)_-$ or $(SR)_-$ is satisfied.\\
\indent Put
$$B:=\{\bar u\in H^1 \mid \|\mathbf{P}\bar u\|_{L^2} \leq R_0,\;\; \|\mathbf{Q}\bar u\|_{H^1}\leq R_\infty \}.
$$
Taking $\delta$ smaller if necessary we may assume that
\begin{equation}\label{25092017-2315}
\delta R_{0}^{2}<\alpha.\end{equation}
Then, for any $\lma\in [\lma_0-\delta,\lma_0+\delta]$, $B$ is an isolating neighborhood for the semiflow $\Phi^\lambda$. To see this, suppose to the contrary that there is
$\bar u\in \mathrm{Inv}_{\Phi^\lambda}(B)\cap \partial B$. Hence there is a solution $u:\R\to B$ of $\Phi^\lambda$ through $\bar u$, i.e. $\bar u=u(0)$. Since $u$ is bounded, we have $\|{\Q}\bar u\|_{H^1}<R_\infty$ in view of \eqref{02112017-2001}. Therefore $\|\mathbf{P}\bar u\|_{L^2}=R_0$. Let $\bar u=\bar v+\bar w$, where
$\bar v:={\P}\bar u$ and $\bar w:={\Q}\bar u$.  Then $\bar v\in X_0\setminus B_{L^2}(0,R_0)$ and $\bar w\in N$. By Lemma \ref{25092017-2310},
$$
\frac{1}{2}\left.\frac{d}{d t} \|\mathbf{P}u(t)\|_{L^2}^{2} \right|_{t=0}=
(\lma-\lma_0)\|{\P}u(0)\|_{L^2}^2 + \la\mathbf{P}u(0), \mathbf{F}(u(0))\ra_{L^2}=(\lambda-\lambda_0)R_0^2+\la \bar v,{\F}(\bar v+\bar w)\ra_{L^2}.
$$
Due to \eqref{25092017-2315} and \eqref{27092017-2022} (or \eqref{27092017-20220})
$$\frac{1}{2}\left.\frac{d}{dt}\|\mathbf{P}u(t)\|_{L^2}^2 \right|_{t=0} >0\quad
\left(\text{or}\;\; \frac{1}{2}\left. \frac{d}{d t}\|\mathbf{P}u(t)\|_{L^2}^2\right|_{t=0} <0\right).
$$
This contradicts the assumption  $u(\R)\subset B$ and proves that $B$ is an isolating neighborhood for the semiflows $\Phi^\lambda$, $\lambda\in [\lma_0-\delta,\lma_0+\delta]$. Using the continuation property (H4) of the homotopy index, we obtain
\begin{equation}\label{27092017-1210}
h(\Phi^{\lma_0-\delta}, K_{\lma_0-\delta}) = h (\Phi^{\lma_0+\delta},K_{\lma_0+\delta})
\end{equation}
where $K_\lma:= \mathrm{Inv}_{\Phi^{\lma}} (B)$ for $\lma\in [\lma_0-\delta,\lma_0+\delta]$.\\
We also claim that, for $\lma\in[\lma_0-\delta, \lma_0+\delta]$, one has
\be\label{212311042018}
K_\lma = K(\Phi^\lma).
\ee
Indeed, the inclusion $K_\lma\subset K(\Phi^\lma)$ is self-evident. Conversely, any bounded full solution  $u:\R\to H^1(\R^N)$ of $\Phi^\lma$ satisfies \eqref{02112017-2001}. Therefore if $u$ leaves $B$, then for some $t\in\R$ we have
$\|\mathbf{P}u(t)\|_{L^2}>R_0$. Put
$t_-:=\inf \{t\in\R \mid \| \mathbf{P} u(t) \|_{L^2}>R_0\}$ and
$t_+:= \sup \{t\in\R \mid \| \mathbf{P} u(t) \|_{L^2}>R_0 \}$. In view of \eqref{08072016-0939} and the fact that $R_0\geq R_\infty>r$ we see that $-\infty<t_- <t_+<+\infty$. It is clear that $\|\mathbf{P}u(t_\pm)\|_{L^2} = R_0$ and
$$
\|\mathbf{P} u(t) \|_{L^2} <R_0 \mbox{ for all } t\in (-\infty,t_-)\cup (t_+,+\infty),
$$
which means that
\be\label{212111042018}
\left.\frac{d}{d t} \|\mathbf{P}u(t)\|_{L^2}^{2} \right|_{t=t_-}\!\!\!\! \geq 0
\ \mbox{ and }\  \left.\frac{d}{d t} \|\mathbf{P}u(t)\|_{L^2}^{2} \right|_{t=t_+} \!\!\!\!\leq 0.
\ee
But on the other hand, as before,
$$
\frac{1}{2}\left.\frac{d}{d t} \|\mathbf{P}u(t)\|_{L^2}^{2} \right|_{t=t_\pm}=
(\lma-\lma_0)\|{\P}u(t_\pm)\|_{L^2}^2 + \la\mathbf{P}u(0), \mathbf{F}(u(0))\ra_{L^2}=(\lambda-\lambda_0)R_0^2+\la \bar v,{\F}(\bar v+\bar w)\ra_{L^2}
$$
which together with \eqref{27092017-2022} (or \eqref{27092017-20220}) yields
$$
\frac{1}{2}\left.\frac{d}{d t} \|\mathbf{P}u(t)\|_{L^2}^{2} \right|_{t=t_\pm}>0 \ \ \  \left(\mbox{or } \frac{1}{2}\left.\frac{d}{d t} \|\mathbf{P}u(t)\|_{L^2}^{2} \right|_{t=t_\pm} <0\right).
$$
This contradicts one of the inequalities in \eqref{212111042018} and shows \eqref{212311042018}. Therefore, by Proposition \ref{13122016-1923}, one has
$$
h(\Phi^{\lambda_0\pm\delta},K_{\lambda_0\pm\delta})=
h(\Phi^{\lambda_0\pm\delta},K(\Phi^{\lambda_0\pm\delta}))=
\Sigma^{k(\lambda_0\pm\delta)}.
$$
and this together with \eqref{27092017-1210} leads to a contradiction, since $k(\lambda_0+\delta)-k(\lambda_0-\delta)=\dim X_0>0$.
\end{proof}

%Observe that, for $\lma=\lma_0\pm\delta$, one has $K_\lma = K(\Phi^\lma)$. Indeed, the inclusion $K_\lma\subset K(\Phi^\lma)$ is self-evident; conversely any bounded full solution of  $\Phi^\lma$  takes values in $B$ since $R_0 \geq R_\infty\geq R$ and, hence, $K(\Phi^\lma)\subset K_\lma$. This together with Proposition \ref{13122016-1923} implies that
%$$h(\Phi^{\lambda_0\pm\delta},K_{\lambda_0\pm\delta})=
%h(\Phi^{\lambda_0\pm\delta},K(\Phi^{\lambda_0\pm\delta}))=
%\Sigma^{k(\lambda_0\pm\delta)}.$$
%This is a contradiction since $k(\lambda_0+\delta)-k(\lambda_0-\delta)=\dim X_0>0$.
%\end{proof}


\begin{thebibliography}{99}
\bibitem{Adams} R. A. Adams, J.J. Fournier, \emph{Sobolev spaces}, Academic Press 2003.

\bibitem{Agueh} M. Agueh, {\em Gagliardo-Nirenberg inequalities involving the gradient $L^2$-norm}, C. R. Acad. Sci. Paris, Ser. I 346 (2008), 757--762.

\bibitem{Arrieta} J. M. Arrieta, R. Pardo, and A. Rodr\'iguez-Bernal, \emph{Equilibria and global dynamics of a problem with bifurcation from infinity}, J. Differential Equations, 246 (2009), 2055--2080.

%\bibitem{Bao} W. Bao, \emph{The Nonlinear Schr\"odinger Equation and Applications in Bose-Einstein Condensation and Plasma Physics}, in \emph{Dynamics in Models of Coarsening, Coagulation, Condensation and Quantization}, Lecture Notes Ser., Natl. University Singapore, World Scientific 2011, 141--239.

\bibitem{Bartolo} P. Bartolo, V. Benci, D. Fortunato, \emph{Abstract critical point theorems and applications to some nonlinear problems with ,,strong" resonance at infinity}, Nonlinear Anal. 7 (1983), 981--1012.

\bibitem{Benci} V. Benci, D. Fortunato, \emph{Variational Methods in Nonlinear Field Equations}, Springer-Verlag, Berlin 2014.

\bibitem{Cesari} L. Cesari, R. Kannan, \emph{An abstract existence theorem at resonance}, Proc. Amer. Math. Soc. 63 (1977), 221--225.

\bibitem{Dlotko} J. Cholewa, T. D{\l}otko, \emph{Global Attractors in Abstract Parabolic Problems}, Cambridge University Press, 2000.

\bibitem{Chiap} R. Chiappinelli, D. G. de Figueiredo, \emph{Bifurcation from infinity and multiple solutions for an elliptic system}, Differential and Integral Equations, 6 (1993),  757--771.

\bibitem{Cwiszewski-RL} A. \'Cwiszewski, R. Lukasiak, {\em Forced periodic solutions for nonresonant parabolic equations on $\R^N$}, \url{http://arxiv.org/pdf/1404.0256.pdf}.

\bibitem{Dancer} E. N. Dancer, \emph{A note on bifurcation from infinity}, Quart. J. of Math. 25 (1974), 81--84.

\bibitem{Engel} K. J. Engel, R. Nagel, \textit{One-Parameter Semigroups for Linear Evolution Equations}, Springer-Verlag, Berlin 2000.

\bibitem{Evequoz} G. Ev\'equoz, C. A. Stuart, \emph{Hadamard differentiability and bifurcation}, Proc. R. Soc. Edin. A, 137 (2007), 1249--1285.

\bibitem{Fonda} A. Fonda, M. Garrione, \emph{Nonlinear Resonance: a Comparison Between Landesman-Lazer and Ahmad-Lazer-Paul Conditions}, Advanced Nonl. Studies 11 (2011), 391--404.

\bibitem{Gamez} J. L. G\'amez and J. F. Ruiz-Hidalgo, \emph{A detailed analysis on local bifurcation from infinity for nonlinear elliptic problems}, J. Math. Anal. Appl., 338 (2008), 1458--1468.

\bibitem{Genoud} F. Genoud, \emph{Global bifurcation for asymptotically linear Schr\"odinger equations}, Nonlin. Diff. Eq. Appl. 20 (2013), 23--35.

\bibitem{Gossez} D. G. de Figueiredo, J.-P. Gossez, {\em Strict monotonicity of eigenvalues and unique continuation}, Comm. Partial Diff. Eq. 17 (1992), 339--346.

\bibitem{Hanche} H. Hanche-Olsen, H. Holden, \textit{The Kolmogorov-Riesz compactness theorem},  Expositiones Mathematicae 28 (2010), 385--394

\bibitem{Hempel} R. Hempel, J. Voigt, \textit{On the $L_p$-spectrum of Schr\"odinger operators},
Journal Math, Anal. Appl. 121 (1987), 138--159.

\bibitem{Henry} D. Henry, \emph{Geometric Theory of Semilinear Parabolic Equations}, Springer Verlag, 1981.

%\bibitem{Hislop} P. Hislop, I. Sigal, \emph{Introduction to Spectral Theory. With Applications to Schrödinger Operators}, Springer-Verlag 1996. New York.

\bibitem{Kokocki} P. Kokocki, {\em Connecting orbits for nonlinear differential equations at resonance},   J. Diff. Eq. 255 (2013), no. 7, 1554--1575.

\bibitem{Kras} M. A. Krasnoselskii, \emph{Topological Methods in the Theory of Nonlinear Integral Equations}, Macmillan, New York, 1965.

\bibitem{Kr-Sz}  W. Kryszewski, A. Szulkin, {\em Bifurcation from infinity for an asymptotically linear Schr\"{o}dinger equation}, J. Fixed Point Theory Appl. 16 (2014), no. 1--2, 411--435.

\bibitem{DLi} C. Li, D. Li, Z. Zhang, \emph{Dynamic Bifurcation from Infinity of Nonlinear Evolution Equations}, SIAM J. Appl. Dyn. Syst., 16 (2017), 1831--1868.

\bibitem{Mawhin} J. Mawhin, K. Schmitt, \emph{Landesman-Lazer type problems at an eigenvalue of odd multiplicity}, Results Math., 14 (1988),  138--146.

\bibitem{Mawhin_Chiap} R. Chiappinelli, J. Mawhin, R. Nugari, \emph{Bifurcation from infinity and multiple solutions for some Dirichlet problems with unbounded nonlinearities}, Nonlinear Anal. TMA, 18 (1992), 1099--1112.

\bibitem{Nirenberg} L. Nirenberg, {\em On elliptic partial differential equations}, Ann. Sc. Norm. Pisa 13 (1959), 116--162.

\bibitem{Pazy} A. Pazy, \emph{Semigroups of linear operators and applications to partial differential equations}, Springer Verlag 1983.

\bibitem{Feng} X.-F. Pang, Y.-P. Feng, \emph{Quantum mechanics in nonlinear systems}, World Scientific, Singapore 2005.

\bibitem{Persson} A. Persson, \emph{Bounds for the discrete part of the spectrum of the semi-bounded Schr\"odinger operator}, Math. Scand. 8 (1960), 143--154.

\bibitem{Prizzi-FM} M. Prizzi, {\em On admissibility of parabolic equations in $\R^N$}, Fund. Math. 176 (2003), 261--275.

\bibitem{Prizzi} M. Prizzi, {\em Averaging, Conley index continuation and reccurent dynamics in almost-periodic parabolic equations}, J. Differential Equations 210 (2005), 429--451.

\bibitem{Rabier} P. J. Rabier, \emph{Bifurcation in weighted Sobolev spaces}, Nonlinearity 21 (2008), 841--856.

\bibitem{Rabinowitz} P. H. Rabinowitz, {\em On bifurcation from infinity}, J. Differential Equations 14 (1973), 462--475.

\bibitem{ReedSimon} M. Reed, B. Simon, \textit{Methods of modern mathematical physics IV: Analysis of operators}, Academic Press, 1980.

\bibitem{rybakowski} K. P. Rybakowski, \textit{The homotopy index and partial differential equations}, Universitext, Springer-Verlag, Berlin, 1987.

\bibitem{rybakowski-TAMS} K. P. Rybakowski, \textit{On the homotopy index for infinite-dimensional semiflows}, Trans. Am. Math. Soc. 269 (1982), 351--382.

\bibitem{Simon} B. Simon, \emph{Schr\"odinger semigroups}, Bull. Amer. Math. Soc. 7 (1982), 447--526.

\bibitem{Schechter} M. Schechter, \emph{Spectra of partial diffrential operators}, North-Holland 1986.

\bibitem{Schmitt_Wang} K. Schmitt, Z. Q. Wang, \emph{On bifurcation from infinity for potential operators}, Differential and Integral Equations, 4 (1991), 933--943.

\bibitem{Schmud} K. Schm\"udgen, \textit{Unbounded Self-adjoint Operators on Hilbert Spaces}, GTM 265, Springer 2012.

\bibitem{Sulem} C. Sulem, P.-L. Sulem, \emph{The Nonlinear Schrodinger Equation}, Springer-Verlag, New York 1999.

\bibitem{Stuart} C. A. Stuart, \emph{Bifurcation at isolated singular points of the Hadamard derivative}, Proc. R. Soc. Edin. A, 144 (2014), 1027--1065

\bibitem{Stuart1} C. A. Stuart, {\em Asymptotic bifurcation and second order elliptic equations on $\R^N$}, Ann. Inst. H. Poincar\'{e} Anal. Non Lin\'{e}aire,
Vol. 32,  Issue 6 (2015), 1259--1281.

\bibitem{Toland} J. F. Toland, \emph{Bifurcation and asymptotic bifurcation for non-compact non- symmetric gradient operators}, Proc. R. Soc. Edin. A, 73 (1975), 137--147.

\bibitem{Ward1} J. R. Ward Jr., \emph{A global continuation theorem and bifurcation from infinity for infinite-dimensional dynamical systems}, Proc. R. Soc. Edin. A, 126 (1996),  725--738.

\bibitem{Ward2} J. R. Ward Jr., \emph{Bifurcating continua in infinite dimensional dynamical systems and applications to differential equations}, J. Differential Equations, 125 (1996), 117--132.



\end{thebibliography}
\end{document}